\documentclass[english]{article}
\usepackage[utf8]{inputenc}
\usepackage{geometry}
\usepackage{amsmath}
\usepackage{amsthm}
\usepackage{amssymb}
\usepackage{dsfont}
\usepackage[hidelinks]{hyperref}
\usepackage{bm} 
\usepackage{graphicx}
\newcommand{\simeqd}{\mathrel{\rotatebox[origin=c]{-90}{$\simeq$}}}

\usepackage[textsize=footnotesize,textwidth=0.85in]{todonotes}

\theoremstyle{definition} \newtheorem{defn}{Definition}[section]
\theoremstyle{definition} \newtheorem{example}[defn]{Example}
\theoremstyle{remark} \newtheorem{rem}[defn]{Remark}
\theoremstyle{plain} \newtheorem{prop}[defn]{Proposition}
\theoremstyle{plain} \newtheorem{lem}[defn]{Lemma}
\theoremstyle{plain} \newtheorem{cor}[defn]{Corollary}
\theoremstyle{definition} \newtheorem{notation}[defn]{Notation}

\theoremstyle{plain} \newtheorem*{thm*}{Theorem}
  
\theoremstyle{plain} \newtheorem{thm}{Theorem}
\theoremstyle{plain}

\begin{document}
\global\long\def\DR#1#2{{\rm DR}_{#1}\left(#2\right)}

\title{Meromorphic differentials, twisted DR cycles and quantum integrable hierarchies}

\author{Xavier Blot and Paolo Rossi}
\maketitle
\begin{abstract}
We define twisted versions of the classical and quantum double ramification hierarchy construction based on intersection theory of the strata of meromorphic differentials in the moduli space of stable curves and $k$-twisted double ramification cycles for $k=1$, respectively, we prove their integrability and tau symmetry and study their connection. We apply the construction to the case of the trivial cohomological field theory to find it produces the KdV hierarchy, although its relation to the untwisted case is nontrivial. The key role of the KdV hierarchy in controlling the intersection theory of several natural tautological classes translates this relation into a series of remarkable identities between intersection numbers involving psi-classes, Hodge classes, Norbury's theta class and the strata of meromorphic differentials.
\end{abstract}
\begin{quotation}
\tableofcontents{}
\end{quotation}

\section{Introduction}

\subsection{Overview}



The last years have seen remarkable effort in the study of the double ramification (DR) hierarchies, a family of integrable hierarchies parametrized by cohomological field theories and constructed using intersection theory of the DR cycle in the moduli space of stable curves. They were introduced in \cite{Buryak15} in the classical setting, and a quantization was defined in \cite{BR2016}. These constructions play an important role in shedding light on the interplay between intersection theory of the moduli space of curves and integrable hierarchies. This provides a novel geometric approach to integrability of evolutionary PDEs and their quantization in the Hamiltonian case, leading for example to an entirely new class of recursion equations for their symmetries \cite{buryak2016recursion,BR2016}. In an opposite direction, these techniques lead to new insight on the cohomology of the moduli space of curves, see for example the  so-called DR/DZ equivalence in \cite{BGR2019} and its generalizations in \cite{buryak2022tautological}, where integrability considerations are used to produce a novel infinite family of tautological relations involving the double ramification cycle. These tautological relations were recently proved in \cite{blot2024strong,blot2024master}. In parallel, a new system of tautological relation was proposed in \cite{blot2024stable}, this time motivated by the geometric representation of tau functions of the quantum DR hierarchies. These relations involve the double ramification cycle, but this time, they also include the so-called Chiodo class. They are proved in \cite{blot2024rooted}. Finally, to mention a few more directions, the DR hierarchies are related to various topics, including modularity \cite{van2022quantum, van2024quantum} and Hurwitz theory \cite{blot2022quantum}.


The goal of this paper is to initiate the study of a twisted version of the double ramification hierarchies, where the twist consists in replacing the usual double ramification cycle with either the fundamental class of strata of meromorphic differentials or the $k=1$-twisted double ramification cycles. Both of these classes represent a way to compactify the locus of meromorphic differentials in the space of smooth stable curves to its Deligne-Mumford compatification. These generalizations are based on similar ideas and integrability and other properties like tau-symmetry, which we prove in this paper, depend on similar cohomological relations satisfied by the corresponding cycles, in particular the splitting property when intersected with boundary strata or psi-classes.

We then study the classical and quantum hierarchy associated to the trivial CohFT. Here we find a somewhat surprising connection with the traditional, untwisted, DR hierarchy, i.e. the (quantum) KdV hiearchy.  We also find an explicit relation between the meromorphic differential and twisted DR hierarchy via a simple change of dependant variables. All this translates into a series of equalities of intersection numbers of independent interest computing Hodge integrals on the strata of meromorphic differentials. The study of its tau function provides more general relations with intersection numbers for other tautological classes, namely monomials in psi-classes and Norbury's class.

\subsection{Plan of the paper}

After defining the necessary algebraic setting of differential polynomials, local functionals, Poisson structures and their quantization, and their formal Laurent series counterpart in Section $2$ we introduce the meromorphic differential and twisted DR hierarchies in Section $3$, where we prove their integrabilty and tau-symmetry as well.

In Section $5$ we compute the meromorphic differential hierarchy for the trivial CohFT, showing first that it is described in terms of differential polynomials and local functionals and identifying it with a quantization of the KdV hiearchy.

Section $6$ is devoted to applications, in particular relations between the meromorphic differential tau-structure and the classical and quantum Witten-Kontsevich and Brezin-Gross-Witten tau functions of (quantum) KdV.

Chapter $7$ and $8$ collects the longer proofs of all of our main results.

\subsection{Acknowledgments}

The development of this project, in particular  the main theorem, owes much to computer experiments performed using the \texttt{admcycles} package \cite{delecroix2022admcycles} to compute various intersection numbers with DR cycles and strata of differentials. 

We specially acknowledge A. Buryak for a key observation which went into the proof of the main theorem in the quantum setting. We are grateful to A. Sauvaget for insightful discussions on strata of differentials and his results on that topic. We also thank D. Zvonkine and R. Tessler for helpful discussions. 

X.B was partially supported by the ISF grant 335/19 in the group of Ran Tessler, and by the Netherland Organization for Scientific Research. P. R. is supported by the University of Padova and is affiliated to the INFN under the national project
MMNLP and to the INdAM group GNSAGA.

\section{Deformation quantization of a formal Poisson structure}

In this section we first introduce a formal Poisson structure and then describe a deformation quantization for it. This formal Poisson structure coincides with the formal Poisson structure of the DR hierarchy in the appropriate setting. 

\begin{notation} 
\label{notation: V}
We fix once and for all $V$ a vector space of dimension $N\geq1$ endowed with a symmetric nondegenerate bilinear form $\eta \in (V^*)^{\otimes 2}$. We fix a basis $\left(e_{1},\dots,e_{N}\right)$ of $V$. Let $\eta_{\alpha\beta}:=\eta\left(e_{\alpha},e_{\beta}\right)$ and denote by $\eta^{\alpha\beta}$ the entries of the inverse matrix $((\eta_{\alpha\beta})_{1 \leq \alpha,\beta\leq N})^{-1}$.
\end{notation}

\subsection{Formal Poisson structure}

We want to define a space of functions on the (formal) loop space of the vector space $V$. An element of this formal loop space can be heuristically thought of as a map $u:S^{1}\rightarrow V$ from the unit circle $S^{1}\subset\mathbb{C}$, but in fact is a formal Laurent series with values in $V$, so that each of its compenents $u^{\alpha}(x)$ for $1\leq\alpha\leq N$, along the basis $\{e_\alpha\}_{1\leq \alpha\leq N}$ is only defined as a formal Laurent series
\[
u^{\alpha}\left(x\right)=\sum_{m\in\mathbb{Z}}q_{m}^{\alpha}\left(ix\right)^{m}
\]
of the formal variable $x$. In this context, functions and Poisson brackets on the space of such formal maps are expressed either as power series of the Laurent coefficients $q_{*}^{*}$, or, in more special cases, as local functionals whose densities are power series of $u^{*}\left(x\right)$ and their derivatives with respect to $x$.

\subsection{Differential polynomials}

Let $u^\alpha_k$, for $1\leq \alpha\leq N$, $k\in \mathbb{Z}_{\geq 0}$, and $\epsilon$ be formal variables. We denote $u^\alpha:=u^\alpha_0$. A \emph{differential polynomial} is an element of the algebra $\mathcal{A}_{u}=\mathbb{C}\left[\left[u^{*}_0\right]\right]\left[u_{>0}^{*}\right]\left[\left[\epsilon\right]\right]$. We make $\mathcal{A}_u$ into a graded algebra by assinging ${\rm deg}\left(u_{k}^{\alpha}\right)=k$ and ${\rm deg}\left(\epsilon\right)=-1$.\\

We denote by $A^{[d]}$ the degree $d$ part of a graded vector space $A$, notice that $\mathcal{A}_u^{[d]} = \mathbb{C}[[u^*_*,\epsilon]]^{[d]}$ with respect to the grading of the formal variables described above.\\

As outlined in the heuristic picture described above, we would like to describe differential polynomials using another set of variables, namely the coefficients of a formal Laurent series representation of the ``map'' $u$. This is made precise by considering new formal variables $q_{m}^{\alpha}$, for $1\leq \alpha\leq N$, $m \in \mathbb{Z}$ and $x$, and the algebra morphism $\phi:\mathcal{A}_u\rightarrow\mathcal{B}:=\mathbb{C}\left[q_{>0}^{*}\right]\left[\left[q_{\leq0}^{*}\right]\right]\left[\left[ix,\left(ix\right)^{-1}\right]\right]\left[\left[\epsilon\right]\right]$ defined on generators as   
\begin{equation}
u_{s}^{\alpha} \mapsto \sum_{m\in\mathbb{Z}}\underset{m^{\underline{s}}}{\underbrace{m\left(m-1\right)\cdots\left(m-s+1\right)}}q_{m}^{\alpha}i^{m}x^{m-s}, \qquad \epsilon \mapsto \epsilon. \label{eq: u_s in q variables}
\end{equation}
The symbol $m^{\underline{s}}$ is called a \emph{falling factorial}. The map $\phi$ is injective and we denote by $\mathcal{A}_{q}$ its image. We interpret this map as a change of variables between $u$-variables and $q$-variables and sometimes write \eqref{eq: u_s in q variables} as an equality. The elements of $\mathcal{A}_{q}$ have the form
\begin{equation}
f\left(x\right)=\sum_{n,l\geq0}\sum_{s=0}^{d_l}\sum_{\substack{m_{1},\dots,m_{n}\in\mathbb{Z}\\
1\leq\alpha_{1},\dots,\alpha_{n}\leq N
}
}f_{n,l,s}^{(\alpha_1,\dots,\alpha_n)}\left(m_{1},\dots,m_{n}\right) q_{m_{1}}^{\alpha_{1}}\cdots q_{m_{n}}^{\alpha_{n}}i^{\sum_{j=1}^{n}m_{j}}x^{\sum_{j=1}^{n}m_{j}-s}\epsilon^l,\label{eq: q variable diff poly}
\end{equation}
where $d_l$ is a nonnegative integer for $l\geq0$, and such that $f_{n,l,s}^{(\alpha_1,\dots,\alpha_n)}$ is either $0$ or a homogenous factorial polynomial of degree $s$ in $\mathbb{C}\left[m_{1},\dots m_{n}\right]$, i.e. a linear combination of factorial monomials $m_{1}^{\underline{s_{1}}}\cdots m_{n}^{\underline{s_{n}}}$ satisfying $\sum_{i=1}^{n}s_{i}=s$. 

\begin{example}
The differential monomial $u_{r_{1}}^{\beta_1}\cdots u_{r_{k}}^{\beta_k}$ corresponds, in $q$-variables, to the element of $\mathcal{A}_q$ of the form \eqref{eq: q variable diff poly} with
\[
f_{n,l,s}^{(\alpha_{1},\dots,\alpha_{n})}\left(m_{1},\dots,m_{n}\right)=\begin{cases}
m_{1}^{\underline{r_{1}}}\cdots m_{k}^{\underline{r_{k}}} & {\rm if}\,\ensuremath{n=k,\,l=0,\,s=\sum_{i=1}^{k}r_{k}},\,\ensuremath{\alpha_{i}=\beta_{i}},\\
0 & {\rm otherwise.}
\end{cases}
\]
Note that $m_{1}^{\underline{r_{1}}}\cdots m_{k}^{\underline{r_{k}}}$ is indeed a homogenous factorial polynomial of degree $\sum_{i=1}^{k}r_{k}$.\\
\end{example}

A \emph{singular differential polynomial} is an element of the algebra $\mathcal{A}_{u}^{{\rm sing}}=\mathbb{C}\left[\left[u^{*}\right]\right]\left[u_{>0}^{*},x^{-1}\right]\left[\left[\epsilon\right]\right]$. We make $\mathcal{A}_{u}^{{\rm sing}}$ into a graded algebra by assinging ${\rm deg}\left(u_{k}^{\alpha}\right)=k,$ ${\rm deg}\left(\epsilon\right)=-1$ as above, and ${\rm deg}\left(\frac{1}{x}\right)=1.$\\

The map (\ref{eq: u_s in q variables}) extends to $\mathcal{A}_{u}^{{\rm sing}}$ by mapping $x \mapsto x$ and identifies the algebra $\mathcal{A}_{u}^{{\rm sing}}$ with its image $\mathcal{A}_{q}^{{\rm sing}}$ in $\mathcal{B}$ given by elements of the form (\ref{eq: q variable diff poly}) where $f_{n,l,s}^{(\alpha_1,\dots,\alpha_n)} \in \mathbb{C}\left[m_{1},\dots m_{n}\right]$ is a factorial polynomial of degree $s$. Notice the elementary fact that a polynomial of degree $s$ is uniquely written as a factorial polynomial of degree $s$ and conversely. 

\begin{example}
The singular differential monomial $u_{r_{1}}^{\beta_1}\cdots u_{r_{k}}^{\beta_k} x^{-j}$ with $j\geq 0$ corresponds in $q$-variables to
\[
f_{n,l,s}^{(\alpha_{1},\dots,\alpha_{n})}\left(m_{1},\dots,m_{n}\right)=\begin{cases}
m_{1}^{\underline{r_{1}}}\cdots m_{k}^{\underline{r_{k}}} & {\rm if}\,\ensuremath{n=k,\,l=0,\,s=\sum_{i=1}^{k}r_{k}}+j,\,\ensuremath{\alpha_{i}=\beta_{i},}\\
0 & {\rm otherwise,}
\end{cases}
\]
and indeed $m_{1}^{\underline{r_{1}}}\cdots m_{k}^{\underline{r_{k}}}$ is a polynomial of degree $\sum_{i=1}^{k}r_{k} =s-j \leq s$ and hence also of degree $s$.
\end{example}

\begin{example}
Consider the singular differential polynomial, depending only linearly on the $u$ variables, of the form $\sum_{r,j\geq 0} c_{\alpha,r,j} u^\alpha_r x^{-j}$, where both sums are finite and $c_{\alpha,r,j} \in \mathbb{C}$. It corresponds, in $q$-variables, to the expression
$$\sum_{s \geq 0}\sum_{m\in \mathbb{Z}} \left(\sum_{r=0}^s c_{\alpha,r,s-r} m^{\underline{r}}\right) q^\alpha_m i^m x^{m-s},$$
where, indeed, $\sum_{r=0}^s c_{\alpha,r,s-r} m^{\underline{r}}$ is the general factorial polynomial of degree $s$.
\end{example}

To summarize, we have
\[
\begin{array}{ccccc}
\mathcal{A}_{u} & \subset & \mathcal{A}_{u}^{{\rm sing}}\\
\simeqd &  & \simeqd\\
\mathcal{A}_{q} & \subset & \mathcal{A}_{q}^{{\rm sing}} & \subset & \mathcal{B}.
\end{array}
\]

We define the derivation $\partial_{x}:\mathcal{B}\rightarrow \mathcal{B}$ simply as $\partial_x = \frac{\partial}{\partial x}$, the derivative with respect to the formal variable $x$. It restricts well to both $ \mathcal{A}_{q}^{{\rm sing}}$ and $\mathcal{A}_{q}$. Via the above isomorphisms, the derivation $\partial_x$ corresponds to the derivations $\partial_x := \frac{\partial}{\partial x} + \sum_{k \geq 0} u^\alpha_{k+1} \frac{\partial}{\partial u^{\alpha}_k}$ in $\mathcal{A}_u^{{\rm sing}}$ and $\partial_x := \sum_{k \geq 0} u^\alpha_{k+1} \frac{\partial}{\partial u^{\alpha}_k}$ in $\mathcal{A}_u$. Notice that this implies in particular $u^{\alpha}_k = \partial_x^k u^\alpha$.\\

In what follows we will effectively identify $\mathcal{A}_{u}$ and $\mathcal{A}_{q}$ (resp. $\mathcal{A}_{u}^{{\rm sing}}$ and $\mathcal{A}_{q}^{{\rm sing}}$) using the above isomorphisms and even remove the subscripts from the notation.\\

We will carry out most of the constructions in $\mathcal{B}$, but sometimes obtain elements of $\mathcal{A}^{{\rm sing}}$ or $\mathcal{A}$, in which case a $u$-variable representation as (singular) differential polynomials will become available.\\

\subsection{Local functionals}

Let $\overline{\mathcal{B}} := \mathbb{C}\left[q_{>0}^{*}\right]\left[\left[q_{\leq0}^{*}\right]\right]\left[\left[\epsilon\right]\right]$. The algebra  $\overline{\mathcal{B}}$ plays the role of algebra of functions on the formal loop space of $V$ in our heuristic picture. We define a map $\int \cdot\ dx : \mathcal{B} \to \overline{\mathcal{B}}$ by
\[
\int fdx={\rm Coef}_{\left(ix\right)^{-1}}\left(f-f\vert_{q_*^*=0}\right).
\]
We will also use the notation $\overline{f}:=\int fdx$. We have ${\rm Im}\,(\partial_{x}:\mathcal{B}\to \mathcal{B})\oplus\frac{\mathbb{C}[[\epsilon]]}{x}\subset {\rm Ker}\int\cdot\ dx$. The kernel of the restircted map $\int\cdot\,dx:\mathcal{A} \rightarrow\overline{\mathcal{B}}$ is exactly ${\rm Im}\, (\partial_{x}:\mathcal{A}\to\mathcal{A})\oplus\mathbb{C}[[\epsilon]]$, thus we will identify the image $\int \mathcal{A}dx$ with the quotient $\mathcal{A}/({\rm Im}\, (\partial_{x}:\mathcal{A}\to\mathcal{A})\oplus\mathbb{C}[[\epsilon]])$, which is called the space of \emph{local functionals}, and the map $\int\cdot\, dx: \mathcal{A} \to \int\mathcal{A}dx$ itself corresponds to the projection onto the quotient.
Similarly, the kernel of the restricted map $\int \cdot\ dx: \mathcal{A}^{{\rm sing}} \to \overline{\mathcal{B}}$ is exactly ${\rm Im}\, (\partial_{x}:\mathcal{A}^{{\rm sing}}\to\mathcal{A}^{{\rm sing}})\oplus\mathbb{C}[[\epsilon]]\oplus\frac{\mathbb{C}[[\epsilon]]}{x}$, thus we will identify the image $\int \mathcal{A}^{{\rm sing}} dx$ with the quotient $\mathcal{A}^{{\rm sing}}/({\rm Im}\,(\partial_{x}:\mathcal{A}^{{\rm sing}}\to\mathcal{A}^{{\rm sing}})\oplus\mathbb{C}[[\epsilon]]\oplus\frac{\mathbb{C}[[\epsilon]]}{x})$, which is called the space of \emph{singular local functionals}, , and the map $\int\cdot\, dx: \mathcal{A}^{{\rm sing}} \to \int\mathcal{A}^{{\rm sing}}dx$ itself corresponds to the projection onto the quotient.

\subsection{Poisson bracket}\label{subsec: Poisson bracket}

Consider the bilinear map $\left\{ \cdot,\cdot\right\} :\mathcal{B}\times\overline{\mathcal{B}}\rightarrow\mathcal{B}$ defined by 
\[
\left\{ f,\overline{g}\right\} =\sum_{k\in\mathbb{Z}}ik\eta^{\alpha\beta}\frac{\partial f}{\partial q_{k-1}^{\alpha}}\frac{\partial\overline{g}}{\partial q_{-k-1}^{\beta}}.
\]
Via the map $\int \cdot\ dx:\mathcal{B} \to \overline{\mathcal{B}}$, the above brackets descends to the Poisson brackets $\{\cdot, \cdot\}:\overline{\mathcal{B}}\times\overline{\mathcal{B}}\rightarrow\overline{\mathcal{B}}$, and a simple computation shows that it can be written, for $\overline{f}, \overline{g} \in \overline{\mathcal{B}}$, as
\[
\left\{ \overline{f},\overline{g}\right\} =\int \left(\frac{\delta\overline{f}}{\delta u^{\alpha}}\eta^{\alpha\beta}\partial_{x}\frac{\delta\overline{g}}{\delta u^{\beta}}\right)dx,
\]
where the variational derivative operator $\frac{\delta}{\delta u^\alpha}:\overline{\mathcal{B}}\to \mathcal{B}$ is defined as 
\[
\frac{\delta\overline{f}}{\delta u^{\alpha}}:=\sum_{m\in\mathbb{Z}}\left(ix\right)^{-m-1}\frac{\partial\overline{f}}{\partial q_{m}^{\alpha}}.
\]
The variational derivative restricts to maps $\frac{\delta}{\delta u^\alpha}:\int \mathcal{A}^{{\rm sing}} dx\to \mathcal{A}^{{\rm sing}}$ and $\frac{\delta}{\delta u^\alpha}:\int \mathcal{A} dx \to \mathcal{A}$ and can be written in terms of the $u$-variables as
$$\frac{\delta \overline{f}}{\delta u^\alpha}=\sum_{s\geq0}\left(-\partial_{x}\right)^{s}\frac{\partial f}{\partial u_{s}^{\alpha}},$$
where $f \in \mathcal{A}^{{\rm sing}}$ (or $f \in \mathcal{A}$) is any representative of the equivalence class $\overline{f}$, and a simple computation shows that, when $f \in \mathcal{A}^{{\rm sing}}$ and $\overline{g} \in \int \mathcal{A}^{{\rm sing}} dx$ (or when $f \in \mathcal{A}$ and $\overline{g} \in \int \mathcal{A}dx$ ), the above brackets can be written as
\[
\left\{ f,\overline{g}\right\} =\sum_{s\geq0}\frac{\partial f}{\partial u_{s}^{\alpha}}\eta^{\alpha\beta}\partial_{x}^{s+1}\left(\frac{\delta\overline{g}}{\delta u^{\beta}}\right).
\]

\begin{rem}
The algebra of differential polynomials $\mathcal{A}_{u}$ and its quotient $\mathcal{A}/({\rm Im}\, (\partial_{x}:\mathcal{A}\to\mathcal{A})\oplus\mathbb{C}[[\epsilon]])$, as well as the Poisson structure above, are the same as the ones encountered in the construction of the standard DR hierarchies (see e.g. \cite{Buryak15,BR2016}). However, in the context of the DR hierarchies a change of variables corresponding to a Fourier series expansion of $u\left(x\right)$ is used, as opposed to the Laurent series expansion encoded by the map \eqref{eq: u_s in q variables}.
\end{rem}

\subsection{Deformation quantization\label{subsec:Deformation-quantization}}

\subsubsection{Star product and commutator in the $q$-variables}

We now describe a deformation quantization of the Poisson brackets we defined above. We follow the same procedure as \cite{BR2016}, namely we first extend the algebras $\overline{\mathcal{B}}$ to $\overline{\mathcal{B}}\left[\left[\hbar\right]\right]$ and then introduce a star product between its elements.\\

Let $\overline{f},\overline{g}\in \overline{\mathcal{B}}\left[\left[\hbar\right]\right]$. We define the star product
\[
\overline{f}\star \overline{g}:=\overline{f}\exp\left(\sum_{k>0}i\hbar k\eta^{\alpha\beta}\overleftarrow{\frac{\partial}{\partial q_{k-1}^{\alpha}}}\overrightarrow{\frac{\partial}{\partial q_{-k-1}^{\beta}}}\right)\overline{g},
\]
where the arrows specify onto which element, $\overline{f}$ or $\overline{g}$, the corresponding derivative acts. We denote by $\left[\overline{f},\overline{g}\right]=\overline{f}\star \overline{g}-\overline{g}\star \overline{f}$ the commutator of the star product and call it the quantum bracket.\\

As it's well known, this Moyal-type star product can be seen as naturally arising from normal ordering of the $q$-variables, we refer to \cite{BR2016,eliashberg2010introduction} for more details. This star product is an associative deformation of the standard commutative product in $\overline{\mathcal{B}}$ and the first order term in $\hbar$ of its commutator is given by the Poisson bracket, that is 
\[
\overline{f}\star \overline{g}=\overline{f}\overline{g}+\mathcal{O}\left(\hbar\right)
\]
and
\[
\left[\overline{f},\overline{g}\right]=\hbar\left\{ \overline{f},\overline{g}\right\} +\mathcal{O}\left(\hbar^{2}\right).
\]

\subsubsection{Commutator in $u$-variables}

Similarly to what we did for the Poisson bracket, we define a lift $ [\cdot,\cdot]:\mathcal{B}[[\hbar]]\times\overline{\mathcal{B}}[[\hbar]] \to \mathcal{B}[[\hbar]]$ of the above quantum bracket with respect to the natural extension of the integration map, $\int \cdot\, dx:\mathcal{B}[[\hbar]]\to \overline{\mathcal{B}}[[\hbar]]$, namely
\[
[f,\overline{g}]:=f\exp\left(\sum_{k>0}i\hbar k\eta^{\alpha\beta}\overleftarrow{\frac{\partial}{\partial q_{k-1}^{\alpha}}}\overrightarrow{\frac{\partial}{\partial q_{-k-1}^{\beta}}}\right)\overline{g} - \overline{g}\exp\left(\sum_{k>0}i\hbar k\eta^{\alpha\beta}\overleftarrow{\frac{\partial}{\partial q_{k-1}^{\alpha}}}\overrightarrow{\frac{\partial}{\partial q_{-k-1}^{\beta}}}\right)f.
\]

Next we want to express the restriction to $\mathcal{A}^{{\rm sing}}\left[\left[\hbar\right]\right]\times \int\mathcal{A}^{{\rm sing}}\left[\left[\hbar\right]\right]dx $ of this lifted commutator in terms of $u$-variables. An element of $\mathcal{A}^{{\rm sing}}\left[\left[\hbar\right]\right]$ (resp. $\mathcal{A}\left[\left[\hbar\right]\right]$) will still be called a singular differential polynomial (resp. differential polynomial). We make $\mathcal{A}^{{\rm sing}}\left[\left[\hbar\right]\right]$ and $\mathcal{A}\left[\left[\hbar\right]\right]$ into graded algebras by assigning  ${\rm deg}\left(u_{k}^{\alpha}\right)=k,$ ${\rm deg}\left(\epsilon\right)=-1$, ${\rm deg}\left(\frac{1}{x}\right)=1$ as above, and ${\rm deg\,}\hbar=-2.$
\begin{prop}
\label{prop: commutator formula}Let $f,g\in\mathcal{A}^{{\rm sing}}\left[\left[\hbar\right]\right]$. Then
\begin{align*}
	\left[f\left(x\right),\overline{g}\right]&= \sum_{n\geq1}\left(-i\right)^{n-1}\frac{\hbar^{n}}{n!} \times\\
	&\times \sum_{\substack{s_{1},\dots,s_{n}\geq0\\
r_{1},\dots,r_{n}\geq0
}
}\frac{\partial^{n}f}{\partial u_{s_{1}}^{\alpha_1}\cdots\partial u_{s_{n}}^{\alpha_n}}\left(-1\right)^{R}\frac{\prod_{i=1}^{n}\eta^{\alpha_i\beta_i}\left(s_{i}+r_{i}+1\right)!}{\left(S+R+2n-1\right)!}\partial_{x}^{S+R+2n-1}\frac{\partial^{n}g}{\partial u_{r_{1}}^{\beta_1}\cdots\partial u_{r_{n}}^{\beta_n}},
\end{align*}
where we have denoted $S=\sum_{i=1}^{n}s_{i}$ and $R=\sum_{i=1}^{n}r_{i}$.
\end{prop}
Notice that, if ${\rm deg }(f)= d_1$ and ${\rm deg}(g) = d_2$, it follows from this formula that ${\rm deg}\left(\left[f,\overline{g}\right]\right) =d_{1}+d_{2}-1$.\\
\begin{rem}
\label{rem: comparison commutator DR and MD}
When restricted to $\mathcal{A}[[\hbar]]\times\int\mathcal{A}[[\hbar]]dx$, we find that this commutator is different from the commutator $[\cdot, \cdot]^{{\rm DR}}$ for the DR hierarchies, see formula $\left(1.2\right)$ of \cite{BR2016}. Explicitly, let $f,g \in \mathcal{A}$ with ${\rm deg }(f)= d_1$ and ${\rm deg}(g) = d_2$, then $\left[f,\overline{g}\right]$ coincides with the top degree part of $\left[f,\overline{g}\right]^{{\rm DR}}$,
\begin{align}
	\left[f,\overline{g}\right]^{{\rm DR}}=\left[f,\overline{g}\right]+{\rm terms\,of\,degree\,lower\,than}\,d_{1}+d_{2}-1. \label{eq: link commutator DR and MD}
\end{align}
\end{rem}

\subsubsection{Proof of Proposition \ref{prop: commutator formula}}

In this section we prove the commutator formula of Proposition~\ref{prop: commutator formula}. 

\paragraph{Ehrhart polynomials.}

Let $n\geq1$ be a positive integer and let $d_{1},\dots,d_{n}\geq 0$ be nonnegative integers. For all nonnegative integers $A\geq0$, we study

\[
C^{d_{1},\dots,d_{n}}\left(A\right)=\sum_{\substack{a_{1},\dots,a_{n}\in\mathbb{Z}_{\geq0}\\
a_{1}+\cdots+a_{n}=A
}
}a_{1}^{\underline{d_{1}}}\cdots a_{n}^{\underline{d_{n}}}.
\]

\begin{lem}
\label{lem: Ehrhart polynomials} $C^{d_{1},\dots,d_{n}}\left(A\right)$ is a polynomial in $A$ of degree $\sum_{i=1}^{n}d_{i}+n-1$ given by 
\[
C^{d_{1},\dots,d_{n}}\left(A\right)=\frac{d_{1}!\cdots d_{n}!}{\left(\sum d_{i}+n-1\right)!}\left(A+n-1\right)^{^{\underline{\sum d_{i}+n-1}}}.
\]
\end{lem}
\begin{proof}
Let
\[
F_{n}\left(z\right)=\sum_{a\geq n}a^{\left(n\right)}z^{a}.
\]
We have 
\[
F_{d_{1}}\left(z\right)\cdots F_{d_{n}}\left(z\right)=\sum_{A\geq\sum d_{i}}C^{d_{1},\dots,d_{n}}\left(A\right)z^{A}
\]
and obviously $C^{d_{1},\dots,d_{n}}\left(A\right)=0$ when $A<\sum d_{i}$. Moreover, we have
\[
F_{n}\left(z\right)=z^{n}\left(\frac{d}{dz}\right)^{n}\frac{1}{1-z}=\frac{n!z^{n}}{\left(1-z\right)^{n+1}}.
\]
Thus 
\begin{align*}
F_{d_{1}}\left(z\right)\cdots F_{d_{n}}\left(z\right) & =d_{1}!\cdots d_{n}!\frac{z^{\sum d_{i}}}{\left(1-z\right)^{\sum d_{i}+n}}.
\end{align*}
Finally, using $\frac{1}{\left(1-z\right)^{d}}=\sum_{j\geq0}\frac{\left(j+d-1\right)^{^{\underline{d-1}}}}{\left(d-1\right)!}z^{j}$ we get 
\begin{align*}
F_{d_{1}}\left(z\right)\cdots F_{d_{n}}\left(z\right) & =d_{1}!\cdots d_{n}!\sum_{j\geq0}\frac{\left(j+\sum d_{i}+n-1\right)^{^{\underline{\sum d_{i}-1}}}}{\left(\sum d_{i}+n-1\right)!}z^{j+\sum d_{i}}\\
 & =d_{1}!\cdots d_{n}!\sum_{A\geq\sum d_{i}}\frac{\left(A+n-1\right)^{^{\underline{\sum d_{i}-1}}}}{\left(\sum d_{i}+n-1\right)!}z^{A}.
\end{align*}
This proves the proposition.
\end{proof}

\paragraph{Computing $f\star\overline{g}$. }

When needed, we will stress the dependence of $f \in \mathcal{B}[[\hbar]]$ (resp. $f\in \mathcal{A}^{{\rm sing }}[[\hbar]] \subset \mathcal{B}[[\hbar]]$ or $f\in \mathcal{A}[[\hbar]] \subset \mathcal{B}[[\hbar]]$ ) on the formal variable $x$ by enriching our notations to $f(x) \in \mathcal{B}^x[[\hbar]]$ (resp. $f(x) \in  \mathcal{A}^{{\rm sing },x}[[\hbar]]\subset \mathcal{B}^x[[\hbar]] $ or $f(x)\in \mathcal{A}^x[[\hbar]] \subset \mathcal{B}^x[[\hbar]]$). Accordingly, we define, for $f(x) \in \mathcal{B}^x[[\hbar]]$ and $g(y) \in \mathcal{B}^y[[\hbar]]$,
$$f(x) \star g(y) := f(x)\exp\left(\sum_{k>0}i\hbar k\eta^{\alpha\beta}\overleftarrow{\frac{\partial}{\partial q_{k-1}^{\alpha}}}\overrightarrow{\frac{\partial}{\partial q_{-k-1}^{\beta}}}\right)g(y),$$ where this formula should be interpreted as a Laurent series in both formal the variables $x$ and $y$. An expression of $f\left(x\right)\star\overline{g}$ is obtained as a corollary of the following proposition.
\begin{prop}
\label{prop: star product f and g}Let $f(x)\in \mathcal{A}^{{\rm sing},x}[[\hbar]]$ and $g(y)\in \mathcal{A}^{{\rm sing},y}[[\hbar]]$. Then 
\begin{align}
f\left(x\right)\star g\left(y\right) & =\sum_{n\geq0}\left(-i\right)^{n-1}\frac{\hbar^{n}}{n!}\sum_{\substack{s_{1},\dots,s_{n}\geq0\\
r_{1},\dots,r_{n}\geq0
}
}\frac{\partial^{n}f}{\partial u_{s_{1}}^{\alpha_1}\cdots\partial u_{s_{n}}^{\alpha_n}}\left(x\right)\frac{\partial^{n}g}{\partial u_{r_{1}}^{\alpha_1}\cdots\partial u_{r_{n}}^{\beta_n}}\left(y\right)\label{eq: star product in u}\\
 & \quad\times\left(-1\right)^{R}\frac{\prod_{k=1}^{n}\eta^{\alpha_k\beta_k}\left(s_{k}+r_{k}+1\right)!}{\left(S+R+2n-1\right)!}\partial_{x}^{S+R+2n-1}\delta_{+}\left(x-y\right),\nonumber 
\end{align}
where $R=\sum r_{j}$ and $S=\sum s_{j}$ and $\delta_{+}\left(x-y\right)=\sum_{a\geq1}\frac{\left(ix\right)^{a-1}}{\left(iy\right)^{a}}$. 
\end{prop}
Let $f\left(y\right)=\sum_{m\in\mathbb{Z}}f_{m}y^{m}$ be a Laurent series. We have $\int_{S^{1}}\delta_{+}\left(x-y\right)f\left(y\right)dy=\left(f\left(x\right)\right)_{+}$ where $\left(f\left(x\right)\right)_{+}=\sum_{m\geq0}f_{m}x^{m}$. Thus, integrating with respect to $y$ formula~(\ref{eq: star product in u}) and using integration by parts several times we get the following expression. 
\begin{cor}
We have
\begin{align*}
f\left(x\right)\star\overline{g} & =\sum_{n\geq0}\left(-i\right)^{n-1}\frac{\hbar^{n}}{n!}\sum_{\substack{s_{1},\dots,s_{n}\geq0\\
r_{1},\dots,r_{n}\geq0
}
}\frac{\partial^{n}f}{\partial u_{s_{1}}^{\alpha_1}\cdots\partial u_{s_{n}}^{\alpha_n}}\left(x\right)\left(-1\right)^{R}\frac{\prod_{k=1}^{n}\eta^{\alpha_k \beta_k}\left(s_{k}+r_{k}+1\right)!}{\left(S+R+2n-1\right)!}\\
 & \quad\times\left(\partial_{x}^{S+R+2n-1}\left(\frac{\partial^{n}g}{\partial u_{r_{1}}^{\beta_1}\cdots\partial u_{r_{n}}^{\beta_n}}\right)\left(x\right)\right)_{+}.
\end{align*}
\end{cor}
\begin{proof}
[Proof of the Proposition \ref{prop: star product f and g}]When acting on a singular differential polynomial, we have 
\[
\frac{\partial}{\partial p_{m}^\alpha}=\sum_{s\geq0}\partial_{x}^{s}\left(ix\right)^{m}\frac{\partial}{\partial u_{s}^\alpha}.
\]
Thus the definition of the star product translates to	

\begin{align*}
f\left(x\right)\star g\left(y\right) & =\sum_{n\geq0}\frac{\hbar^{n}}{n!}\sum_{\substack{s_{1},\dots,s_{n}\geq0\\
r_{1},\dots,r_{n}\geq0
}
}\frac{\partial^{n}f}{\partial u_{s_{1}}^{\alpha_1}\cdots\partial u_{s_{n}}^{\alpha_n}}\left(x\right)\frac{\partial^{n}g}{\partial u_{r_{1}}^{\beta_1}\cdots\partial u_{r_{n}}^{\beta_n}}\left(y\right)\\
 & \times\prod_{k=1}^{n}\left(\eta^{\alpha_k\beta_k}\sum_{m>0}\partial_{x}^{s_{k}}\left(ix\right)^{m-1}im\partial_{y}^{r_{k}}\left(iy\right)^{-m-1}\right).
\end{align*}
The purpose of the rest of the proof is to show following equality
\begin{equation}
\prod_{k=1}^{n}\left(\sum_{m>0}\partial_{x}^{s_{k}}\left(ix\right)^{m-1}im\partial_{y}^{r_{k}}\left(iy\right)^{-m-1}\right) 
	=
	\frac{\prod_{k=1}^{n}\left(s_{k}+r_{k}+1\right)!}{\left(S+R+2n-1\right)!}\left(-1\right)^{R}\left(-i\right)^{n-1}\partial_{x}^{S+R+2n-1}\delta_{+}\left(x-y\right).\label{eq: delta_plus}
\end{equation}
First using $\partial_{x}^{s}\left(ix\right)^{m}=m^{\underline{s}}i^{s}\left(ix\right)^{m-s}$ and similarly for the variable $y$ we rewrite each sum on the LHS as 
\[
\sum_{m>0}\partial_{x}^{s_{k}}\left(ix\right)^{m-1}im\partial_{y}^{r_{k}}\left(iy\right)^{-m-1}=\left(-1\right)^{r_{k}}i^{s_{k}+r_{k}+1}\sum_{m>0}\left(m+r_{k}\right)^{\underline{s_{k}+r_{k}+1}}\left(ix\right)^{m-1-s_{k}}\left(iy\right)^{-m-1-r_{k}}.
\]
Thus we rewrite the LHS of Eq.~(\ref{eq: delta_plus}) as 
\begin{align}
{\rm LHS} & =\sum_{M>0}\left(-1\right)^{R}i^{S+R+n}\left(\sum_{\substack{m_{1},\dots,m_{n}>0\\
\sum_{j=1}^{n}m_{j}=M
}
} \prod_{k=1}^n \left(m_{k}+r_{k}\right)^{\underline{s_{k}+r_{k}+1}}\right)\left(ix\right)^{M-n-S}\left(iy\right)^{-M-n-R}\label{eq: LHS equals blabla}
\end{align}
We now study the term in parenthesis. After the change of indices in summations $m_{i}+r_{i}=a_{i}$, we see that this term is the polynomial studied in Lemma~\ref{lem: Ehrhart polynomials}:
\begin{align*}
\sum_{\substack{m_{1},\dots,m_{n}>0\\
\sum_{j=1}^{n}m_{j}=M
}
} \prod_{k=1}^n \left(m_{k}+r_{k}\right)^{\underline{s_{k}+r_{k}+1}}& =C^{s_{1}+r_{1}+1,\dots,s_{n}+r_{n}+1}\left(M+R\right)\\
 & \overset{{\rm Lemma\,\ref{lem: Ehrhart polynomials}}}{=}\frac{\prod_{k=1}^{n}\left(s_{k}+r_{k}+1\right)!}{\left(S+R+2n-1\right)!}\left(M+R+n-1\right)^{^{\underline{S+R+2n-1}}}
\end{align*}
Substituting this expression in Eq.~(\ref{eq: LHS equals blabla}), we see that the factorial polynomial arises as a $x$-derivative $\partial_{x}^{S+R+2n-1}$, we get
\[
{\rm LHS}=\sum_{M>0}\left(-1\right)^{R}\left(-i\right)^{n-1}\frac{\prod_{k=1}^{n}\left(s_{k}+r_{k}+1\right)!}{\left(S+R+2n-1\right)!}\partial_{x}^{S+R+2n-1}\left(ix\right)^{M+R+n-1}\left(iy\right)^{-M-n-R}.
\]
After the change of index in the sum $A=M+R+n$ and remarking that the $x$-derivatives kills the first terms of the sum, we can rewrite this expression using the function $\delta_{+}\left(x-y\right)$. We obtain the RHS of Eq.~\ref{eq: delta_plus}. This proves the statement.
\end{proof}

\paragraph{Computing the commutator. } We repeat the same computation for $-g\left(y\right)\star f\left(x\right)$. We get a similar expression, but this time the series $\delta_{+}\left(x-y\right)$ is replaced by the series $\delta_{-}\left(x-y\right)=\sum_{a\leq0}\frac{\left(ix\right)^{a-1}}{\left(iy\right)^{a}}$. This series satisfies $\int f\left(y\right)\delta_{-}\left(x-y\right)dy=f_{-}\left(x\right)=\sum_{m\leq-1}f_{m}x^{m}$ and we get a formula for $-\overline{g}\star f\left(x\right)$. Adding the two contribution of $f\left(x\right)\star\overline{g}$ and of $-\overline{g}\star f\left(x\right)$ one obtain the formula of the commutator of Proposition~\ref{prop: commutator formula}. 

\section{The meromorphic differential hierarchy}

\subsection{Strata of differentials}

Let $g$ and $n$ be two nonnegative integers satisfying $2g-2+n>0$. We denote by $\mathcal{M}_{g,n}$ the moduli space of smooth curves of genus $g$ with $n$ marked points, and denote by $\overline{\mathcal{M}}_{g,n}$ the Deligne-Mumford compactification of $\mathcal{M}_{g,n}$. We choose to index the $n$ marked points from $0$ to $n-1$. 

Let $\vec{m}=\left(m_{0},\dots,m_{n-1}\right)\in\mathbb{Z}^{n}$ satisfying 
\[
\sum_{i=0}^{n-1}m_{i}=2g-2.
\]
We define the \emph{stratum of differentials of type $\vec{m}$} by
\[
\mathcal{H}_{g}\left(m_{0},\dots,m_{n-1}\right)=\left\{ \left[C,x_{0},\dots,x_{n-1}\right]\in\mathcal{M}_{g,n}\vert\omega_{C}=\mathcal{O}_{C}\left(\sum_{i=0}^{n-1}m_{i}x_{i}\right)\right\} .
\]
It follows from the line bundle condition that a smooth pointed curve belongs $\mathcal{H}_{g}\left(m_{0},\dots,m_{n-1}\right)$ if and only if there exists a meromorphic differential on the curve whose zeros and poles occur only at the point $x_{0},\ldots,x_{n-1}$ with orders $m_{0},\ldots, m_{n-1}$. When $m_i=0$ the marked point $x_i$ is neither a zero nor a pole, meaning that the meromorphic differential is holomorphic and nonzero in a neighborhood of $x_i$. If $\vec{m}$ contains only nonnegative integers, the stratum is called holomorphic and its codimension is $g-1$, otherwise, the stratum is called meromorphic and its codimension is $g$, see \cite{farkas2018moduli}. Finally, we denote by $\overline{\mathcal{H}}_{g}\left(m_{0},\dots,m_{n-1}\right)$ the closure of $\mathcal{H}_{g}\left(m_{0},\dots,m_{n-1}\right)$ in $\overline{\mathcal{M}}_{g,n}$.

\subsection{Meromorphic differential hierarchy}

The constructions and results of this section work for any Cohomological Field Theory with unit (CohFT). We fix once and for all the notations for a CohFT and refer, for example, to \cite{pandharipande2018cohomological} for a definition.

\begin{notation} 
\label{notation CohFT}
We fix a triple $\left(V,\eta,\bm{1}\right)$, where $(V,\eta)$ is an $N$-dimensional $\mathbb{C}$-vector space together with a nondegenerate bilinear form as already considered in Notation \ref{notation: V}, and  $\bm{1}\in V$ is a special vector. We denote by $c_{g,n}:V^{\otimes n}\rightarrow H^{*}\left(\overline{\mathcal{M}}_{g,n}\right)$ a cohomological field theory (with unit) with phase space $\left(V,\eta,\bm{1}\right)$. Let $\{e_\alpha\}_{1\leq \alpha\leq N}$ be a basis for $V$ and let $\bm{1}=A^\alpha e_\alpha$ for some $A^\alpha \in \mathbb{C}$. Then we use the index $\bm{1}$ as a lower index in component notations for tensors and formal variables associated to that choice of basis to denote a linear combination of the corresponding quantities with index $\alpha$ and coefficients $A^\alpha$, for instance $\frac{\partial}{\partial q^{\bm{1}}_m} := A^\alpha \frac{\partial}{\partial q^\alpha_m}$.
\end{notation}
\begin{defn}
Let $c_{g,n}:V^{\otimes n}\rightarrow H^{*}\left(\overline{\mathcal{M}}_{g,n}\right)$ be a CohFT as introduced in Notation \ref{notation: V}. Fix $d\geq-1$ and $1\leq\alpha\leq N$. For $1\leq \alpha\leq N $ and $d\geq 0$, the \emph{hamiltonian densities of the quantum meromorphic differential (MD) hierarchy} are the elements of $\mathcal{B}\left[\left[\hbar\right]\right]$ defined by 
\begin{align*}
H_{\alpha,d}\left(x\right) & =\sum_{\substack{g,n\geq0\\
2g+n>0
}
}\frac{\left(i\hbar\right)^{g}}{n!}\sum_{\substack{m_{1},\dots,m_{n}\in\mathbb{Z}\\
1\leq\alpha_{1},\dots,\alpha_{n}\leq n
}
}\quad q_{m_{1}}^{\alpha_1}\cdots q_{m_{n}}^{\alpha_n}\left(ix\right)^{\sum_{i=1}^{n}m_{i}-2g}\\
 & \times\left(\int_{\overline{\mathcal{H}}_{g}\left(-1,m_{1},\dots,m_{n},2g-1-\sum_{i=1}^{n}m_{i}\right)}\Lambda\left(\frac{-\epsilon^{2}}{i\hbar}\right)\psi_{0}^{d+1}c_{g,n+2}\left(e_\alpha\otimes \otimes_{i=1}^n e_{\alpha_{i}} \otimes \bm{1}\right)\right),
\end{align*}
where $\Lambda\left(s\right)=1+s\lambda_{1}+\cdots+s^{g}\lambda_{g}$ is Chern polynomial of the Hodge bundle, and where $\psi_{i}$ is the first Chern class of the line bundle over $\overline{\mathcal{M}}_{g,n}$ whose fiber at a marked curve is the cotangent line at its $i$-th marked point.

\begin{notation} 
In the rest of the text, we denote by $H_{d}$ the hamiltonian density of the quantum meromorphic differential hierarchy associated to the trivial CohFT. We also set $h_{d}:=H_{d}\vert_{\hbar=0}$ for its classical counterpart.
\end{notation}
\end{defn}
\begin{rem}
For a generic CohFT, the intersection numbers involved in $H_{\alpha,d}$ are not polynomial in $m_{1},\dots,m_{n}$ but piecewise polynomial. Thus $H_{\alpha,d}\left(x\right)$ is an element of $\mathcal{B}\left[\left[\hbar\right]\right]$ with no interpretation as differential polynomial in the $u$-variables. For the trivial CohFT (see Section~\ref{sec: main results}) and certain CohFTs that we plan to describe in an upcoming paper, these intersection numbers are polynomials of degree $2g$, thus the hamiltonian densities are singular differential polynomials given by
\begin{align}
H_{\alpha,d}\left(x\right) & =\sum_{\substack{g,n\geq0\\
2g+n>0
}
}\frac{\left(i\hbar\right)^{g}}{n!}\sum_{\substack{s_{1},\dots,s_{n}\geq0\\
s_{1}+\cdots+s_{n}\leq2g
}
}\sum_{1\leq\alpha_{1},\dots,\alpha_{n}\leq n}u_{s_{1}}^{\alpha_1}\cdots u_{s_{n}}^{\alpha_n}\frac{\left(-1\right)^{g}}{x^{2g-\sum s_{i}}} \label{:eq: Hd u variables} \\
 & \times\left[m_{1}^{\underline{s_{1}}}\cdots m_{n}^{\underline{s_{n}}}\right]\left(\int_{\overline{\mathcal{H}}_{g}\left(-1,m_{1},\dots,m_{n},2g-1-\sum_{i=1}^{n}m_{i}\right)}\Lambda\left(\frac{-\epsilon^{2}}{i\hbar}\right)\psi_{0}^{d+1}c_{g,n+2}\left(e_\alpha\otimes \otimes_{i=1}^n e_{\alpha_{i}} \otimes \bm{1}\right)\right) \nonumber,
\end{align}
where the notation $\left[m_{1}^{\underline{s_{1}}}\cdots m_{n}^{\underline{s_{n}}}\right]$ means first writing the polynomial in terms of factorial monomials and then extracting the coefficient of $m_{1}^{\underline{s_{1}}}\cdots m_{n}^{\underline{s_{n}}}$. In the special case of the trivial CohFT, we even prove (Theorem~\ref{thm: main theorem}) that the densitis $H_d$, for $d\geq 0$, are actual (not singular) differential polynomials.
\end{rem}
The goal of the rest of this section is to prove that the quantum hierarchy we defined is integrable and tau-symmetric. This is the content of the following proposition. 
\begin{prop}
\label{prop: integrability and tau for meromorphic}The hamiltonian densities satisfy two conditions:
\begin{itemize}
\item (integrability) let $d_{1},d_{2}\geq-1$ and $1\leq\alpha_{1},\alpha_{2}\leq N$, we have 
\[
\left[\overline{H}_{\alpha_{1},d_{1}},\overline{H}_{\alpha_{2},d_{2}}\right]=0,
\]
\item (tau symmetry) let $d_{1},d_{2}\geq0$ and $1\leq\alpha_{1},\alpha_{1}\leq N$, we have
\[
\left[H_{d_{1}-1,\alpha_{1}},\overline{H}_{d_{2},\alpha_{2}}\right]=\left[H_{d_{2}-1,\alpha_{2}},\overline{H}_{d_{1},\alpha_{1}}\right].
\]
\end{itemize}
\end{prop}
\begin{rem}
	This tau-symmetric quantum integrable hierarchy reduces to a classical tau-symmetric integrable hierarchy  with densities given by $h_{d,\alpha}:=H_{d,\alpha}\vert_{\hbar=0}$ and with the Poisson bracket $\{\cdot,\cdot\}$ introduced in Section~\ref{subsec: Poisson bracket}. The integrability and tau-symmetry directly follows from the equations of Proposition~\ref{prop: integrability and tau for meromorphic} upon dividing by $\hbar$ and then evaluating at $\hbar=0$.
\end{rem}

The proof is based on a splitting property of strata of differentials when intersected with a $\psi$-class due to Sauvaget. This formula, as stated in \cite{sauvaget2019cohomology}, only allows to insert $\psi$-classes on zeros. The more general formula allowing $\psi$-classes on poles essentially follows from his work and is stated in a more general context in \cite{costantini2022chern}. We restate here the formula of Sauvaget, with $\psi$-classes on poles. We write this formula as an equality of classes in $H^{*}\left(\overline{\mathcal{M}}_{g,n}\right)$, and write explicitly the summation over graphs  as it is more convenient for our purpose.

 \begin{thm*}
[\cite{sauvaget2019cohomology}, Theorem 6(2)] Fix $s$ and $t$ in $\left\{ 0,\dots,n-1\right\} $ and suppose $m_{s} \in \mathbb{Z}$ and $m_{t} \in \mathbb{Z}$. We have
\begin{align*}
 & \left(\left(m_{s}+1\right)\psi_{s}-\left(m_{t}+1\right)\psi_{t}\right)\overline{\mathcal{H}}_{g}\left(m_{1},\dots,m_{n}\right)=\\
 & \sum_{s\in I,t\in J}\sum_{p\geq1}\sum_{\substack{g_{1}\geq0,g_{2}\geq0\\
g_{1}+g_{2}+p-1=g
}
}\sum_{k_{1},\dots,k_{p}>0}\frac{\prod_{i=1}^{p}k_{i}}{p!}\overline{\mathcal{H}}_{g_{1}}\left(M_{I},k_{1}-1,\dots,k_{p}-1\right)\boxtimes\overline{\mathcal{H}}_{g_{2}}\left(M_{J},-k_{1}-1,\dots,-k_{p}-1\right)\\
 & -\sum_{t\in I,s\in J}\sum_{p\geq1}\sum_{g_{1}+g_{2}+p-1=g}\sum_{k_{1},\dots,k_{p}>0}\frac{\prod_{i=1}^{p}k_{i}}{p!}\overline{\mathcal{H}}_{g_{1}}\left(M_{I},k_{1}-1,\dots,k_{p}-1\right)\boxtimes\overline{\mathcal{H}}_{g_{2}}\left(M_{J},-k_{1}-1,\dots,-k_{p}-1\right),
\end{align*}
where the first sum is taken over $I\sqcup J=\left\{ 0,\dots,n-1\right\} $. We used $M_{I}$ to denote the list $\left(m_{i}\right)_{i\in I}$. The notation $\overline{\mathcal{H}}_{g_{1}}\left(M_{I},k_{1}-1,\dots,k_{p}-1\right)\boxtimes\overline{\mathcal{H}}_{g_{2}}\left(M_{J},-k_{1}-1,\dots,-k_{p}-1\right)$ stands for the class in $\overline{\mathcal{M}}_{g,n}$ obtained by glueing $\overline{\mathcal{H}}_{g_{1}}\left(M_{I},k_{1}-1,\dots,k_{p}-1\right)$ and $\overline{\mathcal{H}}_{g_{2}}\left(M_{J},-k_{1}-1,\dots,-k_{p}-1\right)$ at the $p$ last marked points.
\end{thm*}

\begin{proof}
	The pushforward of the formula of Proposition 8.2 in \cite{costantini2022chern} by the forgetful map forgetting the multi-scale differential and keeping the pointed stable curve generalises Theorem~6(1) of \cite{sauvaget2019cohomology} by allowing $\psi$-classes on poles. Then, just as Theorem~6(1) implies Theorem~6(2) in \cite{sauvaget2019cohomology}, we deduce the version  Theorem~6(2) of \cite{sauvaget2019cohomology} with $\psi$-classes on the poles which is presented here. 
\end{proof}
The proofs of the two assertions of Proposition~\ref{prop: integrability and tau for meromorphic} are similar to the integrability proof of \cite{BR2016} and tau symmetry proof in \cite{BGDR2}. We adapt their ideas in our context.
\begin{proof}
$\left(1\right)$ First, we justify that the hamiltonian $\overline{H}_{\alpha,d}$ is given by 
\begin{align*}
\overline{H}_{\alpha,d} & =\sum_{\substack{g,n\geq0\\
2g+n>0
}
}\frac{\left(i\hbar\right)^{g}}{n!}\sum_{\substack{m_{1},\dots,m_{n}\in\mathbb{Z}\\
\sum m_{i}=2g-1\\
1\leq\beta_{1},\dots,\beta_{n}\leq n
}
}\left(\int_{\overline{\mathcal{H}}_{g}\left(-1,m_{1},\dots,m_{n}\right)}\Lambda\left(\frac{-\epsilon^{2}}{i\hbar}\right)\psi_{1}^{d}c_{g,n+1}\left(e_\alpha\otimes \otimes_{i=1}^n e_{\alpha_{i}} \right)\right)q_{m_{1}}^{\alpha_1}\cdots q_{m_{n}}^{\alpha_n}
\end{align*}
Indeed, extracting the coefficient of $\left(ix\right)^{-1}$ in the generating series $H_{\alpha,d}\left(x\right)$ imposes a zero/pole of order zero, which simply means an unconstrained marked point on the stable curve, as the last marking of each stratum of differentials involved in the formula. Now we have
\begin{equation*}
\begin{split}
\overline{\mathcal{H}}_{g}\left(-1,m_{1},\dots,m_{n},0\right)\Lambda\left(\frac{-\epsilon^{2}}{i\hbar}\right)& c_{g,n+2}\left(e_\alpha\otimes \otimes_{i=1}^n e_{\alpha_{i}} \otimes \bm{1}\right)= \\
& \pi^{*}\left(\overline{\mathcal{H}}_{g}\left(-1,m_{1},\dots,m_{n}\right)\Lambda\left(\frac{-\epsilon^{2}}{i\hbar}\right)c_{g,n+1}\left(e_\alpha\otimes \otimes_{i=1}^n e_{\alpha_{i}}\right)\right),
\end{split}
\end{equation*}
where $\pi:\overline{\mathcal{M}}_{g,n+2}\rightarrow\overline{\mathcal{M}}_{g,n+1}$ is the map forgetting the last marked point. We conclude using the pull-back property of the $\psi$-class. 

Then, we remark that the intersection numbers involved in the commutator
\[
\overline{H}_{\alpha_{1},d_{1}}\star\overline{H}_{\alpha_{2},d_{2}}-\overline{H}_{\alpha_{2},d_{2}}\star\overline{H}_{\alpha_{1},d_{1}}.
\]
correspond to the RHS of Sauvaget formula in case the psi-classes are at simple poles. In that case the LHS of Sauvaget's formula vanishes, and so does the commutator.

$\left(2\right)$ 
Recall that the variational derivative $\frac{\delta}{\delta u^\alpha} : \overline{\mathcal{B}}[[\hbar]] \to \mathcal{B}[[\hbar]]$ is defined as $\frac{\delta \overline{f} }{\delta u^\alpha} = \sum_{m\in\mathbb{Z}}\left(ix\right)^{-m-1}\frac{\partial\overline{f}}{\partial q_{m}^{\alpha}}$. We apply this operator to the integrability condition, which yields
\[
\frac{\delta}{\delta u^{\bm{1}}}\left[\overline{H}_{\alpha_{1},d_{1}},\overline{H}_{\alpha_{2},d_{2}}\right]=\left[\frac{\delta}{\delta u^{\bm{1}}}\overline{H}_{\alpha_{1},d_{1}},\overline{H}_{\alpha_{2},d_{2}}\right]+\left[\overline{H}_{\alpha_{1},d_{1}},\frac{\delta}{\delta u^{\bm{1}}}\overline{H}_{\alpha_{2},d_{2}}\right]=0.
\]
Together with the fact that
\[
\frac{\delta}{\delta u^{\bm{1}}}\overline{H}_{\alpha_{1},d_{1}}=\sum_{m\in\mathbb{Z}}\left(ix\right)^{-m-1}\frac{\partial\overline{H}_{\alpha_{1},d_{1}}}{\partial p_{m}^{\bm{1}}}=H_{\alpha_{1},d_{1}-1},
\]
by the pull-back property of the $\psi$-class, this completes the proof.
\end{proof}

\section{The twisted double ramification hierarchies}

\subsection{Twisted double ramification cycles}

Let $g$ and $n$ be two nonnegative integers satisfying $2g-2+n>0$. Let $\vec{m}=\left(m_{0},\dots,m_{n-1}\right)\in\mathbb{Z}^{n}$ satisfying 
\[
\sum_{i=0}^{n-1}m_{i}=2g-2.
\]
In \cite[Section 1.1]{JPPZ17}, the authors define the $k$-twisted Pixton class ${\rm P}_{g}^{d,k}\left(A\right) \in H^{2d}(\overline{\mathcal{M}}_{g,n})$, for $d,k\geq0$ and $A=\left(a_{0},\dots,a_{n-1}\right)$ such that $\sum_{i=0}^{n-1}a_{i}=k\left(2g-2+n\right)$ via an explicit tautological class formula and then define the \emph{twisted double ramification} \emph{cycle} as 
\[
{\rm DR}_{g}^{1}\left(m_{0},\dots,m_{n-1}\right):=2^{-g}{\rm P}_{g}^{g,1}\left(m_{0}+1,\dots,m_{n-1}+1\right).
\]
In the meromorphic case, this class has the property of reproducing the weighted fundamental class of the moduli space of twisted canonical divisors of \cite{farkas2018moduli} which, in turn, contains as an open set $\mathcal{H}_g(m_0,\ldots,m_{n-1})$, but has in general irreducible components that do not lie in $\overline{\mathcal{H}}_g(m_0,\ldots,m_{n-1})$. Morover the class ${\rm DR}_g^1(m_0,\ldots,m_{n-1})$ is polynomial of degree $2g$ in $m_0,\ldots,m_{n-1}$ as proved in \cite{spelier2024polynomiality,pixton2023polynomiality}.

\subsection{Twisted double ramification hierarchies}
\begin{defn}
Given a CohFT $c_{g,n}:V^{\otimes n}\rightarrow H^{*}\left(\overline{\mathcal{M}}_{g,n}\right)$ with a unit as in Notation \ref{notation: V}, we define the \emph{hamiltonian densities of the quantum ${\rm DR}^{1}$  hierarchy} by
\begin{align*}
H_{\alpha,d}^{{\rm DR}^{1}}\left(x\right) & =\sum_{\substack{g,n\geq0\\
2g+m>0
}
}\frac{\left(i\hbar\right)^{g}}{n!}\sum_{\substack{m_{1},\dots,m_{n}\in\mathbb{Z}\\
1\leq\alpha_{1},\dots,\alpha_{n}\leq n
}
}\quad q_{m_{1}}^{\alpha_1}\cdots q_{m_{n}}^{\alpha_n}\left(ix\right)^{\sum_{i=1}^{n}m_{i}-2g}\\
 & \times\left(\int_{{\rm DR}_{g}^{1}\left(-1,m_{1},\dots,m_{n},2g-1-\sum_{i=1}^{n}m_{i}\right)}\Lambda\left(\frac{-\epsilon^{2}}{i\hbar}\right)\psi_{1}^{d+1}c_{g,n+2}\left(e_\alpha\otimes \otimes_{i=1}^n e_{\alpha_{i}} \otimes \bm{1}\right)\right).
\end{align*}
\end{defn}
Since ${\rm DR}_{g}^{1}\left(m_{1},\dots,m_{n}\right)$ is a polynomial in $m_{1},\dots,m_{n}$ of degree $2g$ as proved in \cite{spelier2024polynomiality,pixton2023polynomiality}, the hamiltonian densities $H_{\alpha,d}^{{\rm DR}^{1}}\left(x\right)$ are singular differential polynomials. 

The quantum ${\rm DR^{1}}$ hierarchies satisfy the same integrability condition and tau symmetry property than the meromorphic differential hierarchy. 
\begin{prop}
The hamiltonian densities satisfies two conditions:
\begin{itemize}
\item (integrability) let $d_{1},d_{2}\geq-1$ and $1\leq\alpha_{1},\alpha_{2}\leq N$, we have 
\[
\left[\overline{H}_{\alpha_{1},d_{1}}^{{\rm DR}^{1}},\overline{H}_{\alpha_{2},d_{2}}^{{\rm DR}^{1}}\right]=0,
\]
\item (tau-symmetry) let $d_{1},d_{2}\geq0$ and $1\leq\alpha_{1},\alpha_{1}\leq N$, we have (tau symmetry)
\[
\left[H_{d_{1}-1,\alpha_{1}}^{{\rm DR}^{1}},\overline{H}_{d_{2},\alpha_{2}}^{{\rm DR}^{1}}\right]=\left[H_{d_{2}-1,\alpha_{2}}^{{\rm DR}^{1}},\overline{H}_{d_{1},\alpha_{1}}^{{\rm DR}^{1}}\right].
\]
\end{itemize}
\end{prop}
\begin{proof}
In \cite[Proposition 3.1]{CSS}, the authors give a splitting formula for the ${\rm DR}^{1}$ cycle. This formula is similar to the Sauvaget splitting formula for the strata of differentials. In particular, the proof of the integrability and the tau symmetry of the ${\rm DR}^{1}$ hierarchy directly follows from the splitting formula of \cite{CSS} using the arguments employed to prove the integrability and tau symmetry for the meromorphic differentials hierarchies.
\end{proof}

\section{Integrable hierarchies associated to the trivial CohFT}
\label{sec: main results}

In this section we state results on the meromorphic differential hierarchy and the twisted DR hierarchy for the trivial CohFT $c_{g,n}(\bm{1}^{\otimes n})=1 \in H^0(\overline{\mathcal{M}}_{g,n})$ for $g,n\geq 0$ such that $2g-2+n>0$, with phase space $V=\langle \bm{1}\rangle$, bilinear form $\eta$ given by $\eta(\bm{1},\bm{1})=1$ and unit $\bm{1}$. Since in this case $N=\dim V  =1$, we suppress the greek indices (usually running from $1$ to $N$) from all of our notations.

\subsection{Polynomial behaviour}

First, the intersection numbers involved in the density $H_d$ satisfy the following property. 

\begin{prop}
\label{prop: polynomiality intersection numbers} Fix $g,n\geq0$ such that $2g-1+n>0$, fix $0\leq l\leq g$ and let $d$ be a nonnegative integer. The function 
\[
\int_{\overline{\mathcal{H}}_{g}\left(2g-2-\sum_{i=1}^{n}m_{i},m_{1},\dots,m_{n}\right)}\psi_{0}^{d}\lambda_{l},
\]
such that $2g-2-\sum_{i=1}^{n}m_{i}<0$, is a polynomial in the variables $m_{1},\dots,m_{n}$ of degree $2g.$
\end{prop}
The proof of this proposition is given in Section~\ref{sec:Applications of FP conjecture}. 

It follows from this proposition that the hamiltonian density $H_{d}$ of the trivial meromorphic differential hierarchy is a singular differential polynomial. This hamiltonian density is actually a differential polynomial (not singular). This follows from the main theorem. 

\subsection{Main theorem}

We now formulate our main result. It describes a link between the intersection numbers involved in $H_d$ and some intersection numbers on the DR cycle. Fix a list of integers $(a_{1},\dots,a_{n})$ such that $\sum_{i=1}^{n}a_{i}=0$, we denote by
\[
{\rm DR}_{g}\left(a_{1},\dots,a_{n}\right)\in H_{2\left(2g-3+n\right)}\left(\overline{\mathcal{M}}_{g,n}\right)
\]
the \emph{double ramification cycle} and refer, for example, to \cite{JPPZ17} for a definition. We recall that intersection numbers of any tautological class with the DR cycle are polynomial in the variables $a_{1},\dots,a_{n}$ of degree $2g$, as proved in \cite{pixton2023polynomiality,spelier2024polynomiality} or in the appendix of \cite{BR2016}.
 By extracting the coefficient of $m_{1}^{\underline{s_{1}}}\cdots m_{n}^{\underline{s_{n}}}$ in a polynomial in $m_{1},\dots,m_{n}$ we mean first writing the polynomial in terms of factorial monomials and then extracting the coefficient of the corresponding monomial.
\begin{thm}
[Main theorem]\label{thm: main theorem} Fix $g,n\geq0$ such that $2g+n>0$, fix an integer $l$ such that $0\leq l\leq g$ and let $d$ be a nonnegative integer. Let $\left(s_{1},\dots,s_{n}\right)$ be a list of nonnegative integers such that $\sum_{i=1}^{n}s_{i}=2g$. Then the coefficient of $m_{1}^{\underline{s_{1}}}\cdots m_{n}^{\underline{s_{n}}}$ in the polynomial
\[
\int_{\overline{\mathcal{H}}_{g}\left(-1,m_{1},\dots,m_{n},2g-2-\sum m_{i}+1\right)}\psi_{0}^{d}\lambda_{l}
\]
vanishes if $\sum_{i \geq 0}^{m}s_i \neq 2g$, and when $\sum_{i \geq 0}^{m}s_i = 2g$ it is given by the coefficient of $a_{1}^{s_{1}}\cdots a_{n}^{s_{n}}$ in the polynomial
\[
\int_{\DR g{0,a_{1},\dots,a_{n}-\sum a_{i}}}\psi_{0}^{d}\lambda_{l}.
\]
In particular, $\int_{\overline{\mathcal{H}}_{g}\left(-1,m_{1},\dots,m_{n},2g-2-\sum m_{i}+1\right)}\psi_{0}^{d+1}\lambda_{l}$ is a homogenous factorial polynomial of degree $2g$.
\end{thm}
The proof of this theorem is given in Section~\ref{sec: proof main theorem}. 

\begin{rem}
	In \cite{CSS} , an explicit expression is given for $\int_{\overline{\mathcal{H}}_{g}\left(-1,m_{1},\dots,m_{n},2g-2-\sum m_{i}+1\right)}\psi_{0}^{d+1}$, and in  \cite{BSSZ} an explicit expression is given for $\int_{\DR g{0,a_{1},\dots,a_{n}-\sum a_{i}}}\psi_{0}^{d+1}$. Both expressions look similar and involve the function $\frac{\sinh\left(z/2\right)}{z/2}$, however it is nontrivial to check the equality of Theorem~\ref{thm: main theorem} from these expressions.
\end{rem}

\subsection{Relation between the MD and DR hierarchies for the trivial CohFT}

In \cite[Corollary 4.3]{BGDR2} a system of hamiltonian densities for the trivial quantum DR hierarchy is defined. One can write these densities as 
\[
H_{d}^{{\rm DR}}=\sum_{\substack{g,n\geq0\\
2g+n>0
}
}\frac{\left(i\hbar\right)^{g}}{n!}\sum_{s_{1},\dots,s_{n}\geq0}\left(-i\right)^{S}\left[a_{1}^{s_{1}}\cdots a_{n}^{s_{n}}\right]\left(\int_{\DR g{0,a_{1},\dots,a_{n}-\sum a_{i}}}\psi_{0}^{d+1}\Lambda\left(\frac{-\epsilon^{2}}{i\hbar}\right)\right)u_{s_{1}}\cdots u_{s_{n}},
\]
for $d \geq -1$, where we used $S=\sum_{j=1}^{n}s_{j}$, and we refer to \cite[Appendix]{blot2022quantum} for a justification of this formulation. According to Proposition~\ref{prop: polynomiality intersection numbers}, one can write $H_d$ as a singular differential polynomial in terms of $u$-variables as in formula (\ref{:eq: Hd u variables}). Thus, Theorem~\ref{thm: main theorem} compares the coefficients of  of $H_d$ and $H_d^{\rm{DR}}$ written with the $u$-variables, in particular an equivalent formulation is  
\[
H_{d}=\left(H_{d}^{{\rm DR}}\right)^{[0]},\quad{\rm for}\,\,d\geq-1,\]
where the square bracket $[0]$ means extracting the degree $0$ in $H_{d}^{{\rm DR}}$, recalling that the degree of a singular differential polynomial is determined by ${\rm deg}\,u_{i}^{\alpha}=i$, ${\rm deg}\,\epsilon=-1$, ${\rm deg}\,\frac{1}{x}=1$, and ${\rm deg}\,\hbar=-2$.  As a consequence we find that $H_d$ is an actual (not singular) differential polynomial. 

In addition, since the highest degree of $H_d^{\rm{DR}}$ is zero and since the quantum bracket of the meromorphic differential hierarchy is the highest degree part of the quantum bracket of the DR hierarchy (see formula (\ref{eq: link commutator DR and MD})), we conclude that,  for the trivial CohFT, the MD hierarchy $(\left(H_{d}\right)_{d\geq0},\left[\cdot,\cdot\right])$ is the highest degree part of the DR hierarchy $(\left(H_{d}\right)_{d\geq0}^{{\rm DR}},\left[\cdot,\cdot\right]^{{\rm DR}})$.

In the classical setting, when $\hbar = 0$, the hamiltonian densities of the classical DR hierarchy $h_d^{\rm{DR}}:=H_{d}^{{\rm DR}}\vert_{\hbar=0}$ are homogenous of degree 0, thus, always for the trivial CohFT, 
\[
h_d=h_{d}^{{\rm DR}},
\]
recalling that $h_d:=H_{d}\vert_{\hbar=0}$. Moreover, the Poisson brackets of the MD and DR hierarchies are identified in $u$-variables. Thus, the classical MD hierarchy $\left((h_d)_{d\geq 0},\left\{ \cdot,\cdot\right\} \right)$ is identified with the classical DR hierarchy $((h_{d}^{{\rm DR}})_{d \geq 0},\left\{ \cdot,\cdot\right\})$, itself identified with the KdV hierarchy by a result of Buryak \cite[Section 4.3.1]{Buryak15}. In particular, the quantum MD hierarchy for the trivial CohFT gives a quantization of KdV different from the one given by the DR hierarchy. 

\subsection{Relation  between the MD and twisted DR hierarchies for the trivial CohFT}

We state a link between the ${\rm DR}^{1}$ hierarchy and the meromorphic differential hierarchies for the trivial CohFT in the quantum setting. 
\begin{prop}
\label{prop: link DR1 and MD}We have
\[
H_{d}^{{\rm DR}^{1}}=H_{d}\vert_{u\left(x\right)=u\left(x\right)+\frac{1}{24}\frac{\epsilon^{2}}{x^{2}}}.
\]
\end{prop}

This proposition is proved in Section~\ref{sec:Applications of FP conjecture}.

\section{\label{subsec:Applications:-tau-functions}Applications of the main theorem: tau functions of KdV}

As a consequence of Theorem~\ref{thm: main theorem}, the classical meromorphic differential hierarchy for the trivial CohFT is identified with the classical trivial DR hierarchy, itself identified with the KdV hierarchy. Thus, following \cite[Section 3.3]{BGDR1}, one can construct tau functions of KdV from the hamiltonian densities $h_{d}$ and the Poisson bracket. In this section, we give formulas for two tau functions of KdV enjoying particularly nice geometrical interpretations of their coefficients: the Witten-Kontsevich tau function and the BGW tau function. This yields non trivial formulas for Hodge integrals over strata of differentials involved in $h_{d}$. 

More generally, in the quantum setting, Theorem~\ref{thm: main theorem} identifies the trivial meromorphic differential hierarchy with the highest degree part of the trivial DR hierarchy. We explain here how the Hodge integral involved in $H_d$ relate to a quantum tau function of KdV called the quantum Witten-Kontsevich series. 

\subsection{The Witten-Kontsevich tau function}

The Witten-Kontsevich tau function is the logarithm of the tau function of KdV associated to  the solution with initial condition $u\left(x\right)=x$. It is proved in \cite{kontsevich1992intersection} that this is a generating series of intersection of monomial of $\psi$-classes. 
\begin{prop}
Fix $g,n\geq0$ such that $2g-1+n>0$ and fix $d_{1},\dots,d_{n}\geq0$. We have
\[
\int_{\overline{\mathcal{M}}_{g,n+1}}\psi_{1}^{d_{1}}\cdots\psi_{n}^{d_{n}}={\rm Coef}_{\epsilon^{2g}}\left\{ \cdots\left\{ h_{d_{1}-1},\overline{h}_{d_{2}}\right\} \cdots,\overline{h}_{d_{n}}\right\} \Bigg\vert_{\substack{x=0\\
q_{m}=\left(-i\right)\delta_{m,1}
}
}.
\]
In particular, for $n=1$ we get
\[
\int_{\overline{\mathcal{M}}_{g,2}}\psi_{1}^{3g-1}=\int_{\overline{\mathcal{H}}_{g}\big(\text{\normalsize$ -1 $},\underset{2g}{\underbrace{1,\dots,1}},\text{\normalsize$ -1 $}\big)}\psi_{0}^{3g-1}\lambda_{g},
\]
and this number is well known to be equal to $\frac{1}{24^g g!}$.
\end{prop}
\begin{rem}
The intersection number of any monomial of $\psi$-classes can be obtained from the numbers $\int_{\overline{\mathcal{M}}_{g,n+1}}\psi_{1}^{d_{1}}\cdots\psi_{n}^{d_{n}}$ via the string equation. 
\end{rem}
\begin{proof}

In \cite[Section 3.3]{BGDR1}, tau functions of the DR hierarchy are constructed from $h_d^{\rm{DR}}$ and the standard hydrodynamic Poisson bracket. Since $h_{d}=h_{d}^{{\rm DR}}$ and the Poisson bracket of the hierarchies are identified in $u$-variables, we carry this construction in the context of the meromorphic differentials hierarchy. As defined in \cite[Section 3.3]{BGDR1}, the coefficient $\epsilon^{2g}t_{0}t_{d_{1}}\cdots t_{d_{n}}$ of the logarithm of the tau function associated to the solution of the hierarchy $u^{{\rm sol}}\left(x,t_{*}^{*}\right)$ with initial condition $u^{{\rm sol}}\left(x,t_{*}^{*}=0\right)=f\left(x,\epsilon\right)$ is given by
\[
\left\{ \left\{ \Omega_{0,d_{1}},\overline{h}_{d_{2}}\right\} ,\dots,\overline{h}_{d_{n}}\right\} \bigg\vert_{u_{k}=\partial_{x}^{k}f\vert_{x=0}},
\]
where $\Omega_{0,d}=h_{d-1}+C$ with $\Omega_{0,d}\vert_{u_{*}=0}=0$. In this formula, the quantity $\left\{ \left\{ \Omega_{0,d_{1}},\overline{h}_{d_{2}}\right\} ,\dots,\overline{h}_{d_{n}}\right\} $ is a differential polynomial written in $u$-variables and we evaluate it at $u_{k}=\partial_{x}^{k}f\vert_{x=0}$. 

First, note that for the trivial CohFT we have $h_{d}\vert_{u_{*}=0}=0$. Indeed, the evaluation $u_{*}=0$ of a differential polynomial is equivalent to $q_{*}=0$, this selects in $h_{d}$ intersection numbers with the strata $\overline{\mathcal{H}}_{g}\left(-1,2g-1\right)$ which exist only for $g>0$ and vanish because of the residue condition. Thus $\Omega_{0,d}=h_{d-1}$.

Now, the trivial DR hierarchy is KdV. Moreover, Kontsevich proved in \cite{kontsevich1992intersection} that the coefficient $\epsilon^{2g}t_{0}t_{d_{1}}\cdots t_{d_{n}}$ of the logarithm of the tau function of KdV associated to the solution starting with the initial condition $u^{{\rm sol}}\left(x\right)=x$ is $\int_{\overline{\mathcal{M}}_{g,n+1}}\psi_{1}^{d_{1}}\cdots\psi_{n}^{d_{n}}$. Translating this initial condition in $q$-variables gives $u\left(x\right)=\sum_{m\in\mathbb{Z}}q_{m}\left(ix\right)^{m}=x$ and thus
\[
q_{m}=\left(-i\right)\delta_{m,1}.
\]
This gives the desired result. 
\end{proof}

\subsection{The BGW tau function}

The Brezin-Gross-Witten (BGW) tau function is the tau function of KdV associated to the solution with initial condition 
\[
u\left(x\right)=\frac{\epsilon^{2}}{8\left(1-x\right)^{2}}.
\]
In \cite{norbury2017new}, Norbury introduces a class $\Theta_{g,n}$ and conjectured that the coefficients of the BGW tau function are given by intersection of monomial of $\psi$-classes with this class. This conjecture was proved in \cite{chidambaram2021relations}. 
\begin{prop}
Fix $g,n\geq0$ such that $2g-1+n>0$ and fix $d_{1},\dots,d_{n}\geq0$. We have
\[
\int_{\overline{\mathcal{M}}_{g,n+1}}\Theta_{g,n+1}\psi_{1}^{d_{1}}\cdots\psi_{n}^{d_{n}}={\rm Coef}_{\epsilon^{2g}}\left\{ \cdots\left\{ h_{d_{1}-1},\overline{h}_{d_{2}}\right\} \cdots,\overline{h}_{d_{n}}\right\} \bigg\vert_{{\rm evaluation}}
\]
with the evaluation at 
\begin{equation}
q_{m}=\begin{cases}
\frac{\epsilon^{2}}{8}\left(-i\right)^{m}\left(m+1\right) & {\rm for\;}m\geq0,\\
0 & {\rm for\;}m<0,
\end{cases}\label{eq: condition BGW}
\end{equation}
and $x=0$. In particular, for $n=1$, we get 

\[
\int_{\overline{\mathcal{M}}_{g,2}}\Theta_{g,2}\psi_{0}^{g-1}=\sum_{\substack{\gamma+n=g\\
2\gamma+n>0
}
}\sum_{\substack{m_{1}+\cdots+m_{n}=2\gamma\\
m_{1},\dots,m_{n}\geq0
}
}\frac{\prod_{i=1}^{n}\left(m_{i}+1\right)}{n!8^{n}}\int_{\overline{\mathcal{H}}_{\gamma\left(-1,m_{1},\dots,m_{n},-1\right)}}\psi_{0}^{g-1}\lambda_{\gamma}
\]
\end{prop}

\begin{rem}
The intersection number of any monomial of $\psi$-class with the Norbury class is obtained from the numbers $\int_{\overline{\mathcal{M}}_{g,n+1}}\Theta_{g,n+1}\psi_{1}^{d_{1}}\cdots\psi_{n}^{d_{n}}$ using the dilaton equation, see \cite{norbury2017new}. 
\end{rem}
\begin{proof}
We use the same justification than for the Witten-Kontsevich tau function. In this case, the initial condition $u\left(x\right)=\frac{\epsilon^{2}}{8\left(1-x\right)^{2}}$ translates in condition~(\ref{eq: condition BGW}) in $q$-variables.
\end{proof}

\subsection{The quantum Witten-Kontsevich tau function}

In \cite{BGDR2}, the authors introduced quantum tau functions in the context of quantum DR hierarchies. The study of the quantum Witten-Kontsevich series, that is the logarithm of the quantum tau function of the trivial DR hierarchy associated to the initial condition $u\left(x\right)=x$, was initiated in \cite{blot2022quantum} and pursued in \cite{blot2024quantum} and \cite{BlotLewanski}. We use the notation 
\[
\left\langle \tau_{d_{1}}\cdots\tau_{d_{n}}\right\rangle _{l,g-l} \quad d_{1},\dots,d_{n}\geq0,\:g\geq0,\:0\leq l\leq g,
\]
for the correlators of the quantum Witten-Kontsevich series, and refer to \cite[Section 1.3]{blot2022quantum} for their definition. 

\begin{prop}
\label{prop: qTau}
Fix $g,n,l\geq0$ such that $2g-1+n>0$ and $0\leq l\leq g$. Fix $d_{1},\dots,d_{n}\geq0$. The number
\[
{\rm Coef}_{\epsilon^{2l}\hbar^{g-l}}\frac{i^{g-l}}{\hbar^{n-1}}\left[\cdots\left[H_{d_{1}-1},\overline{H}_{d_{2}}\right]\cdots,\overline{H}_{d_{n}}\right]\Bigg\vert_{\substack{x=0\\
q_{m}=\left(-i\right)\delta_{m,1}
}
}
\]
equals
\[
\begin{cases}
\ensuremath{\left\langle \tau_{0}\tau_{d_{1}}\cdots\tau_{d_{n}}\right\rangle _{l,g-l}} & {\rm if\,\,\,}\sum_{i=1}^{n}d_{i}=4g-2+n-l,\\
0 & {\rm otherwise.}
\end{cases}
\]
\end{prop}

\begin{rem}
	Interestingly, a correlator $\left\langle \tau_{0}\tau_{d_{1}}\cdots\tau_{d_{n}}\right\rangle _{0,g}$, for $l=0$, is given by a coefficient of one-part double Hurwitz numbers by \cite[Theorem 1]{blot2022quantum}. Moreover the correlators satisfying $\sum_{i=1}^{n}d_{i}=4g-2+n$, that is precisely the correlators produced by the hierarchy of meromorphic differentials, can be assemble to form a solution of the Hirota equations (i.e. the bilinear Hirota form of the Kadomtsev-Petviashvili hierarchy) by a result of \cite{shadrin2008structure,shadrin2007changes}. 
	
	In addition, a general correlator $\left\langle \tau_{0}\tau_{d_{1}}\cdots\tau_{d_{n}}\right\rangle _{l,g-l}$, for any $l\geq 0$, is on the one hand a coefficient of a Gromov-Witten invariant of $\mathbb{CP}^1$ \cite{blot2024quantum}, and on the other hand a coefficient of an intersection number with a Chiodo class as conjectured in \cite{BlotLewanski}.  
	
	It is intriguing that these topics are related to Hodge integrals over strata of differentials by Proposition~\ref{prop: qTau}, and it would be interesting to investigate more these connections.
\end{rem}

\begin{rem}
Once again, correlators without the $\tau_{0}$-insertion can be obtained from the correlators with a $\tau_{0}$-insertion using the string equation \cite[Theorem 2]{blot2022quantum}.
\end{rem}
\begin{proof}
By definition we have
\begin{equation}
	\left\langle \tau_{0}\tau_{d_{1}}\cdots\tau_{d_{n}}\right\rangle _{l,g-l}={\rm Coef}_{\epsilon^{2l}\hbar^{g-l}}\frac{i^{g-l}}{\hbar^{n-1}}\Big[\cdots\Big[H_{d_{1}-1}^{{\rm DR}},\overline{H}_{d_{2}}^{{\rm DR}}\Big]^{{\rm DR}}\cdots,\overline{H}_{d_{n}}^{{\rm DR}}\Big]^{{\rm DR}}\Bigg\vert_{u_{s}=\delta_{s,1}} \label{eq: quantum correlator}
\end{equation}
where the bracket $\left[\cdot,\cdot\right]^{{\rm DR}}$ is the commutator of star product of the DR hierarchy. We explained in Section~\ref{sec: main results} that $H_{d}$ is the top degree part of $H_{d}^{{\rm DR}}$, and justified in Section~\ref{subsec:Deformation-quantization} that the quantum bracket of the meromorphic differential hierarchy is the top degree part of the quantum bracket of the DR hierarchy. Thus, replacing $H_{d}^{{\rm DR}}$ and the DR star product in Eq.~(\ref{eq: quantum correlator}) by $H_{d}$ and the star product of the meromorphic differential hierarchy selects the top level correlators. This top level is characterised by the condition $\sum_{i=1}^{n}d_{i}=4g-2+n-l$ by \cite[Proposition~1.38]{blot2022quantum}. Moreover, the substitution $u_{s}=\delta_{s,1}$ translates into $q_{m}=\left(-i\right)\delta_{m,1}$ and $x=0$ as in the classical setting. 
\end{proof}

\section{\label{sec:Applications of FP conjecture}Proof of Proposition~\ref{prop: polynomiality intersection numbers} and Proposition~\ref{prop: link DR1 and MD}}

In this section, we explain how Proposition~\ref{prop: polynomiality intersection numbers} and Proposition~\ref{prop: link DR1 and MD} follow from a formula conjectured by Farkas and Pandharipande \cite[Conjecture A]{farkas2018moduli}, and proved in \cite{BHPSS2020pixton}. 

The main ingredient is the following lemma. 
\begin{lem}
\label{lem: Conjecture A for intersection numbers}Fix $g,n,l\geq0$ such that $2g+n>0$ and $0\leq l\leq g$. Fix $d\geq0$. Let $m_{1},\dots,m_{n}$ be some integers such that $2g-2-\sum_{i=1}^{n}m_{i}<0$. We have
\begin{align}
\int_{{\rm DR}_{g}^{1}\left(2g-2-\sum_{i=1}^{n}m_{i},m_{1},\dots,m_{n}\right)}\psi_{0}^{d+1}\Lambda\left(\mu\right) & =\sum_{k=0}^{g}\left(\frac{\mu}{24}\right)^{k}\frac{1}{k!}\int_{\overline{\mathcal{H}}_{g-k}\Big(2g-2-\sum_{i=1}^{n}m_{i},m_{1},\dots,m_{n},\underset{k}{\underbrace{-2,\dots,-2}}\Big)}\psi_{0}^{d+1}\Lambda\left(\mu\right),\label{eq: link DR1 and diff}
\end{align}
where $\mu$ is a formal variable and $\Lambda\left(\mu\right)=1+\mu\lambda_{1}+\cdots+\mu^{g}\lambda_{g}$. 
\end{lem}

\begin{proof}
In \cite[Appendix A.4]{farkas2018moduli}, a class 
\[
{\rm H}_{g,\left(m_{0},\dots,m_{n}\right)}
\]
is constructed and we refer to this reference for its definition. The statement of Conjecture A is 
\[
{\rm DR}_{g}^{1}\left(m_{0},\dots,m_{n}\right)={\rm H}_{g,\left(m_{0},\dots,m_{n}\right)},
\]
if there exists a negative integer in the list $m_{0},\dots,m_{n}$. Intersecting this equality with $\psi_{0}^{d+1}$ and $\Lambda\left(\mu\right)$ yields
\begin{equation}
\int_{{\rm DR}_{g}^{1}\left(2g-2-\sum_{i=1}^{n}m_{i},m_{1},\dots,m_{n}\right)}\psi_{0}^{d+1}\Lambda\left(\mu\right)=\int_{{\rm H}_{g,\left(2g-2-\sum_{i=1}^{n}m_{i},\dots,m_{n}\right)}}\psi_{0}^{d+1}\Lambda\left(\mu\right).\label{eq: Conjecture A for int numbers}
\end{equation}
The class ${\rm H}_{g,\left(m_{0},\dots,m_{n}\right)}$ is a weighted sum over star graphs with a strata of meromorphic differentials on the central vertex and a strata of holomorphic differentials on each leaf of the graph. Since the $0$-th marked point corresponds to a pole it lies on the central vertex, so the class $\psi_{0}^{d+1}$ lies on the central vertex. Moreover the full Hodge class being a CohFT, it splits well when intersected with a boundary stratum. Thus, the contribution of a leaf is given by the intersection number of the class $\Lambda\left(\mu\right)$ on a strata of holomorphic differentials. A direct dimension counting shows that the only non vanishing contribution of a leaf is
\[
\int_{\overline{\mathcal{H}}_{g=1}\left(0\right)}\lambda_{1}=\frac{1}{24}.
\]
Thus, a star graph contributing to the RHS of Eq.~(\ref{eq: Conjecture A for int numbers}) is necessarily such that: all marked points are on the central vertex, moreover a leaf can only be of genus $1$, connected by a single edge of twist $1$ to the central vertex and it contributes by $\frac{1}{24}$. This gives the desired conclusion. 
\end{proof}
We now explain how this lemma implies Proposition~\ref{prop: polynomiality intersection numbers} and Proposition~\ref{prop: link DR1 and MD}. 
\begin{proof}
[Proof of Proposition~\ref{prop: polynomiality intersection numbers}]We prove by induction on the genus that
\[
\int_{\overline{\mathcal{H}}_{g}\Big(2g-2-\sum_{i=1}^{n}m_{i},m_{1},\dots,m_{n},\Big)}\psi_{0}^{d+1}\Lambda\left(\mu\right)
\]
is a polynomial of degree $2g$ in $m_{1},\dots,m_{n}$.  This intersection number appears as the term $k=0$ in the sum on the RHS of Eq.~(\ref{eq: link DR1 and diff}). The rest of the terms on the RHS are of smaller genera. Moreover, it is proved by \cite{PixtonZagier} that the cycle ${\rm DR}_{g}^{1}\left(2g-2-\sum_{i=1}^{n}m_{i},m_{1},\dots,m_{n}\right)$ is a polynomial of degree $2g$ in $m_{1},\dots,m_{n}$, and thus the LHS of Eq.~(\ref{eq: link DR1 and diff}) is a polynomial of degree $2g$. We deduce by induction the desired property. 
\end{proof}
\begin{proof}
[Proof of Proposition~\ref{prop: link DR1 and MD}]We first substitute $m_{n}=2g-1-\sum_{i=1}^{n-1}m_{i}$ and $\mu=\frac{-\epsilon^{2}}{i\hbar}$ in Eq.~(\ref{eq: link DR1 and diff}). Then this equality of numbers translate in the following equality of generating series
\[
H_{d}^{{\rm DR}^{1}}=\sum_{l\geq0}\frac{1}{l!}\left(\frac{-\epsilon^{2}}{24}\frac{\partial}{\partial q_{-2}}\right)^{l}H_{d}=\exp\left(\frac{-\epsilon^{2}}{24}\frac{\partial}{\partial q_{-2}}\right)H_{d}.
\]
Moreover, the operator $\frac{\partial}{\partial q_{-2}}$ acting on a singular differential polynomial writes in $u$-variables as $\frac{\partial}{\partial q_{-2}}=\sum_{k\geq0}\partial_{x}^{k}\left(\frac{1}{\left(ix\right)^{2}}\right)\frac{\partial}{\partial u_{k}}$. Thus, using the general identity of power series $\exp\left(t\partial_{x}\right)f\left(x\right)=f\left(x+t\right)$, we obtain 
\[
H_{d}^{{\rm DR}^{1}}=H_{d}\Bigg\vert_{\substack{u_{0}\rightarrow u_{0}+\frac{\epsilon^{2}}{24}\frac{1}{x^{2}}\\
\quad u_{1}\rightarrow u_{1}+\frac{\epsilon^{2}}{24}\partial_{x}\left(\frac{1}{x^{2}}\right)\\
\quad\vdots
}
}=H_{d}\Bigg\vert_{u\left(x\right)\rightarrow u\left(x\right)+\frac{1}{24}\frac{\epsilon^{2}}{x^{2}}}.
\]
\end{proof}

\section{\label{sec: proof main theorem}Proof of the main theorem}

In this section, we prove Theorem~\ref{thm: main theorem}. Equivalently, we prove
\[
H_{d}=\left(H_{d}^{{\rm DR}}\right)^{[0]},\quad{\rm for}\,\,d\geq-1.
\]
where $H_{d}$ is the hamiltonian density of the quantum meromorphic differential hierarchy for the trivial CohFT, and $H_{d}^{{\rm DR}}$ is the hamiltonian density of the quantum trivial DR hierarchy introduced in Section~\ref{sec: main results}. In addition, we recall that the degree of a singular differential polynomial is determined by 
\[
{\rm deg}\left(u_{i}\right)=i,\quad{\rm deg}\left(\epsilon\right)=-1,\quad {\rm deg}\left(\hbar\right)=-2 \;\;{\rm and}\;\;{\rm deg}\left(\frac{1}{x}\right)=1,
\]
and the upper index $[0]$ stands for extracting the degree $0$ in a singular differential polynomial. 

Before explaining the strategy of the proof, we introduce two singular differential polynomials.

\subsection{Preliminaries}

Fix $d\geq0$. We define
\[
G_{d}:=\sum_{\substack{g,n\geq0\\
2g+-1+n>0
}
}\frac{\left( i\hbar \right)^{g}}{n!}\sum_{m_{1},\dots,m_{n}\in\mathbb{Z}}\quad P_{g,d}\left(m_{1},\dots,m_{n}\right)\quad q_{m_{1}}\cdots q_{m_{n}}\left(ix\right)^{\sum_{i=1}^{n}m_{i}-2g},
\]
where $P_{g,d}\left(m_{1},\dots,m_{n}\right)$ is the polynomial in $m_{1},\dots,m_{n}$ given by $\int_{\overline{\mathcal{H}}_{g}\left(2g-2-\sum_{i=1}^{n}m_{i},m_{1},\dots,m_{n}\right)}\psi_{0}^{d} \Lambda\left(\frac{-\epsilon^2}{i\hbar}\right)$ for $2g-2-\sum_{i=1}^{n}m_{i}<0$, where we recall our notation $\Lambda\left(s\right)=1+s\lambda_{1}+\cdots+s^{g}\lambda_{g}$. The polynomiality follows from Proposition~\ref{prop: polynomiality intersection numbers}. We have
\[
\frac{\delta\overline{G}_{d+1}}{\delta u}=\sum_{m\in\mathbb{Z}}\left(ix\right)^{-m-1}\frac{\partial\overline{G}_{d+1}}{\partial q_{m}}=H_{d}
\]
and
\[
\overline{G}_{d}=\overline{H}_{d},
\]
where the last equality follows from a direct calculation using the pull-back property of the $\psi$-class. Thus, the densities $\left(G_{d}\right)_{d\geq0}$ form a second system of hamiltonian densities for the quantum meromorphic differential hierarchy for the trivial CohFT. 

We also introduce, following \cite{BR2016}, hamiltonian densities of the quantum trivial DR hierarchy written in terms of the $u$-variables: 
\[
G_{d}^{{\rm DR}}=\sum_{\substack{g,n\geq0\\
2g-1+n>0
}
}\frac{\left(i\hbar\right)^{g}}{n!}\sum_{s_{1},\dots,s_{n}\geq0}\left[a_{1}^{s_{1}}\cdots a_{n}^{s_{n}}\right]\left(\int_{\DR g{-\sum_{i=1}^{n},a_{1},\dots,a_{n}}}\psi_{0}^{d} \Lambda\left(\frac{-\epsilon^2}{i\hbar}\right) \right)u_{s_{1}}\cdots u_{s_{n}},
\]
for $d\geq0$. We have
\[
\frac{\delta\overline{G_{d+1}^{{\rm DR}}}}{\delta u}=H_{d}^{{\rm DR}},
\]
and they satisfy
\[
\overline{G_{d}^{{\rm DR}}}=\overline{H_{d}^{{\rm DR}}}.
\]

\subsection{Strategy of the proof}

The goal of the rest of the proof is to show that
\begin{equation}
\overline{G}_{d}=\left(\overline{G_{d}^{{\rm DR}}}\right)^{[0]},\quad{\rm for\;}d\geq0.\label{eq: gd equality}
\end{equation}
As a consequence we get
\[
H_{d}=\frac{\delta\overline{G}_{d+1}}{\delta u}=\frac{\delta \left(\overline{G_{d}^{{\rm DR}}}\right)^{[0]}}{\delta u}=\left(\overline{H_{d}^{{\rm DR}}}\right)^{[0]},\quad{\rm for\;}d\geq-1.
\]
which proves the theorem.

The equality (\ref{eq: gd equality}) is obtained as an application of Lemma~\ref{lem: Buryak-Dubrovin lemma} stated below. This lemma, and its proof, is an adaptation of a lemma of Buryak \cite[Lemma 2.4]{Buryak_Hodge15} for singular differential polynomials (instead of differential polynomials) and in the quantum setting. We thank Alexandr Buryak for helping us to prove the quantum version of this lemma.

\begin{lem}
\label{lem: Buryak-Dubrovin lemma}
Suppose we have $\overline{H}=\int\left(\frac{u^{3}}{3!}+O\left(\epsilon,\hbar\right)\right)dx$ and $\overline{Q}=\int\left(Q_{0}+O\left(\epsilon,\hbar\right)\right)dx$ two singular local functionals of homogenous degree $0$ such that 
\[
\left[\overline{Q},\overline{H}\right]=0.
\]
Then $\overline{Q}$ is uniquely determined by $\overline{H}$ and $\overline{Q_{0}}$. In particular, if $\overline{Q_{0}}=0$ then $\overline{Q}=0$.
\end{lem}

Before giving the proof, we explain how formula~(\ref{eq: gd equality}) is deduced from this lemma. We are going to prove the following equalities:
\begin{enumerate}
\item we have 
\[
\overline{G}_{1}=\left(\overline{G_{1}^{DR}}\right)^{\left[0\right]}=\int\left(\frac{u^{3}}{3!}+O\left(\epsilon,\hbar \right)\right),
\]
\item we have
\[
\left[\epsilon^{0}\hbar^{0}\right]\overline{G}_{d}=\left[\epsilon^{0}\hbar^{0}\right]\left(\overline{G_{d}^{DR}}\right)^{\left[0\right]}=\frac{u^{d+2}}{\left(d+2\right)!},\quad d\geq0,
\]
\item we have
\[
\left[\overline{G_{d}},\overline{G_{1}}\right]=0=\left[\left(\overline{G_{d}^{DR}}\right)^{\left[0\right]},\left(\overline{G_{1}^{DR}}\right)^{\left[0\right]}\right],\quad d\geq0.
\]
\end{enumerate}

Once this is established, it follows from Lemma~\ref{lem: Buryak-Dubrovin lemma} that $\overline{G}_{d}$ and the degree $0$ of $\overline{G_{d}^{DR}}$ are uniquely determined by $\overline{G}_{1}=\left(\overline{G_{1}^{DR}}\right)^{\left[0\right]}$ and $\int\frac{u_{0}^{d+2}}{\left(d+2\right)!}dx$, thus they are equal. This is the desired conclusion. 

We first prove the items $1,2$ and $3$. Then we prove Lemma~\ref{lem: Buryak-Dubrovin lemma}. 

\subsection{Proof of item $3$}

We justified in the preliminaries that $\overline{G}_{d}=\overline{H}_{d}$, then the integrability condition of the quantum meromorphic differentials hierarchy gives the first equality. 

The integrability condition of the quantum DR hierarchies, proved in \cite{BR2016}, gives $[ \overline{G_{d}^{DR}},\overline{G_{1}^{DR}} ]^{DR}=0$ for $d \geq 0$. Moreover, each highest degree monomial  in $G_{d}^{DR}$, for $d\geq 0$, is of degree zero. Thus formula~(\ref{eq: link commutator DR and MD}) gives
\[
	0 = \left[\left(\overline{G_{d}^{DR}}\right)^{\left[0\right]},\left(\overline{G_{1}^{DR}}\right)^{\left[0\right]}\right]+ \rm{terms \, of \, lower \, degree},
\]
where $\left[\left(\overline{G_{d}^{DR}}\right)^{\left[0\right]},\left(\overline{G_{1}^{DR}}\right)^{\left[0\right]}\right]$ is of homogeneous differential polynomial of degree $-1$ as one can check using Proposition~\ref{prop: commutator formula}. Thus it vanishes. This prove the second equality of item $3$.

\subsection{Proof of item $2$}

In genus $0$, the strata of meromorphic differentials and the DR cycle are codimension $0$ cycles in the moduli space of curves. Thus they are equal. A direct calculation shows that the genus $0$ contribution of $G_{d}$ and $G_{d}^{{\rm DR}}$ is given by $\frac{u_{0}^{d+2}}{\left(d+2\right)!}$.

\subsection{Proof of item $1$}

We have \cite[Section 3.1.1]{BR2016}
\[
\left(\overline{G_{1}^{{\rm DR}}}\right)^{[0]}=\int\left(\frac{u^{3}}{6}+\frac{\epsilon^{2}}{24}uu_{2}\right)dx.
\]
We now compute $\overline{G}_{1}$. The intersection numbers involved in $G_{1}$ are 
\[
\int_{\overline{\mathcal{H}}_{g}\left(2g-2-\sum_{i=1}^{n}m_{i},m_{1},\dots,m_{n}\right)}\psi_{0}^{1}\lambda_{l},\quad{\rm with}\,\,2g-2-\sum_{i=1}^{n}m_{i}<0,\quad{\rm and}\,\,0\leq l \leq g.
\]
By dimension counting, they can be non zero only for $\left(g,n,l\right)=\left(0,3,0\right),\left(1,2,1\right),\left(2,1,2\right),\left(1,1,0\right)$ and also possibly when $n=0$. However, each $n=0$ contribution to $G_1$ either belongs to the kernel of the integration map, so it does not contribute to $\overline{G_1}$. As we justified in the previous point, the genus $0$ contribution of $G_{1}$ is $\frac{u^{3}}{6}$. The contributions $(2,1,2)$ and $(1,1,0)$ of $\int G_{1}dx$ are respectively
\[
\int_{\overline{\mathcal{H}}_{2}\left(-1,3\right)}\psi_{0}^{1}\lambda_{2} \quad \rm{and}\quad \int_{\overline{\mathcal{H}}_{1}\left(-1,1\right)}\psi_{0}^{1}\lambda_{0},
\]
which vanish. Indeed, in both case the strata of differentials has a unique simple pole, and thus by the residue condition these strata are empty. Finally, the contribution $(1,2,1)$ of $G_{1}$ is computed as follows. Lemma~\ref{lem: Conjecture A for intersection numbers} gives
\[
\int_{\overline{\mathcal{H}}_{g=1}\left(-m_{1}-m_{2},m_{1},m_{2}\right)}\psi_{0}^{1}\lambda_{1}=\int_{{\rm DR}_{g=1}^{1}\left(-m_{1}-m_{2},m_{1},m_{2}\right)}\psi_{0}^{1}\lambda_{1}-\frac{1}{24},
\]
when $m_{1}+m_{2}<0$. Moreover, in genus $1$ the ${\rm DR}$- and ${\rm DR}^{1}$ cycles are equal
\[
{\rm DR}_{g=1}^{1}\left(-m_{1}-m_{2},m_{1},m_{2}\right)={\rm DR}_{g=1}\left(-m_{1}-m_{2},m_{1},m_{2}\right).
\]
This follows from Pixton formula by applying known tautological relations between divisors of $\overline{\mathcal{M}}_{1,n}$. Furthermore, the class $\lambda_{1}$ is supported on the moduli space of genus one curves of compact type, thus we compute the integral using Hain formula \cite{hain2011normal}. We get 
\[
\int_{{\rm DR}_{g=1}\left(-m_{1}-m_{2},m_{1},m_{2}\right)}\psi_{0}^{1}\lambda_{1}=\frac{m_{1}^{2}+m_{2}^{2}}{24}.
\]
Finally, gathering the different contributions and using $m^{2}=m^{\underline{2}}+m^{\underline{1}}$, we obtain
\[
\int_{\overline{\mathcal{H}}_{g=1}\left(-m_{1}-m_{2},m_{1},m_{2}\right)}\psi_{0}^{1}\lambda_{1}=\frac{m_{1}^{\underline{2}}+m_{2}^{\underline{2}}}{24}+\frac{m_{1}^{\underline{1}}+m_{2}^{\underline{1}}}{24}-\frac{1}{24}.
\]
Thus the contribution $(1,2,1)$ of $G_{1}$ is
\[
\left[\epsilon^{2}\hbar^0\right]G_{1}=\frac{uu_{2}}{24}+\partial_{x}\left(\frac{1}{48}\frac{u_{0}^{2}}{x}\right).
\]
We conclude that $\overline{G}_{1}=\left(\overline{G_{1}^{{\rm DR}}}\right)^{[0]}$.

\subsection{Proof of Lemma~\ref{lem: Buryak-Dubrovin lemma}}
The key ingredient is the following lemma.

\begin{lem}
\label{lem: Key technical lemma}Let $P$ be a singular differential polynomial such that $P$ does not depend on $\epsilon$ nor $\hbar$ and is homogenous of positive degree. Suppose that $\left\{ \int Pdx,\int\frac{u^{3}}{6}dx\right\} =0$,
	then
\[
\int Pdx=0.
\]
\end{lem}
Before proving Lemma~\ref{lem: Key technical lemma}, we explain how it implies Lemma~\ref{lem: Buryak-Dubrovin lemma}. Suppose $\overline{Q}^{\left(1\right)}$ and $\overline{Q}^{\left(2\right)}$ are two local functionals of degree $0$ satisfying the hypothesis of Lemma~\ref{lem: Buryak-Dubrovin lemma}, and let 
\[
\overline{P}=\overline{Q}^{\left(1\right)}-\overline{Q}^{\left(2\right)}=\sum_{i,j\geq0}\overline{P_{i,j}}\hbar^{i}\epsilon^{j}.
\]
Suppose that $\overline{P}$ is nonzero, then let $i_{0}$ be the minimal power of $\hbar$ in $\overline{P}$ and let $j_{0}$ be the minimal power of $\epsilon$ in factor of $\hbar^{i_{0}}$ in $\overline{P}$, so that $\overline{P_{i_{0},j_{0}}}\neq0$. Let us verify that $\overline{P_{i_{0},j_{0}}}$ satisfies the hypothesis of Lemma~\ref{lem: Buryak-Dubrovin lemma}. By definition of $\overline{Q}^{\left(1\right)}$ and $\overline{Q}^{\left(2\right)}$, we have $\overline{P_{0,0}}=0$ so that $i_{0}+j_{0}>0$. In addition, since $\overline{Q}^{\left(1\right)}$ and $\overline{Q}^{\left(1\right)}$ are of degree $0$, we get that the degree of $ \overline{P_{i_{0},j_{0}}}$ is positive. Moreover, it cannot contain a term $\alpha/x$ by hypothesis. Finally, since the quantum commutator has the form (see Proposition~\ref{prop: commutator formula})
\[
\left[\cdot,\cdot\right]=\hbar\left\{ \cdot,\cdot\right\} +\hbar^{2}\mu\left(\cdot,\cdot\right),
\]
where $\mu$ is bilinear with respect to $\epsilon$ and $\hbar$, we deduce that 
\[
0=\left[\overline{P},\overline{H}\right]=\hbar^{i_{0}+1}\epsilon^{j_{0}}\left\{ \overline{P}_{i_{0},j_{0}},\int\frac{u^{3}}{3!}\right\} +O\left(\hbar^{i_{0}+2}\right)+O\left(\epsilon^{j_{0}+1}\right),
\]
and then $\left\{ \overline{P}_{i_{0},j_{0}},\int\frac{u^{3}}{3!}\right\} =0$. Hence, by Lemma~\ref{lem: Key technical lemma}, we get $\overline{P_{i_{0},j_{0}}}=0$. This contradicts our minimal assumption and we deduce $\overline{P}=0$. This proves Lemma~\ref{lem: Buryak-Dubrovin lemma}. 

\begin{proof}[Proof of Lemma~\ref{lem: Key technical lemma}]
The proof is a direct adaptation of the proof of \cite[Lemma 2.5]{Buryak_Hodge15} for singular differential polynomials. We expose here the modifications required to generalize Buryak's argument.

First, we verify that the lemma works if ${\rm deg}P=1$. In that case, $P$ can only contain the following terms 
\[
u_{0}^{n}u_{1}\,{\rm for}\,n\geq0,\quad\frac{cste}{x},\quad{\rm and}\quad\frac{u_{0}^{n}}{x}{\rm \,for\,}n>0.
\]
The first and second terms are in the kernel of the integration map. If $P$ contains terms of the third type, it cannot satisfy $\left\{ \int Pdx,\int\frac{u^{3}}{6}dx\right\} =0$. Indeed, it suffices to remark that $\left\{ \int\frac{u_{0}^{n}}{x},\int\frac{u^{3}}{6}dx\right\} =\int n\frac{u_{0}^{n}u_{1}}{x}dx$ and the singular local functionals $\int n\frac{u_{0}^{n}u_{1}}{x}dx$, for $n>0$, are linearly independent. Thus $P$ can only contain terms of the first type and then satisfies $\int Pdx=0$.

Now suppose that ${\rm deg}P\geq2$. We define a lexicographic order on each monomial using the order
\[
\frac{1}{x}<u_{0}<u_{1}<\cdots,
\]
this induces a total order relation between the monomials of $P$. Suppose that $$P=f\left(u\right)\frac{u_{1}^{\alpha_{1}}\cdots u_{m}^{\alpha_{m}}}{x^{n}}+{\rm lower\,lexicographic\,terms},$$ with $f\left(u\right)=\sum_{i\geq0}f_{i}u_{0}^{i}$ and $f_{i}\in\mathbb{C}.$ If $m=0$, then $P=\sum_{i\geq 0}f_{i}\frac{u_{0}^{i}}{x^{n}}$ where $n={\rm deg}P$, but repeating our last argument shows that $\left\{ \int Pdx,\int\frac{u^{3}}{6}dx\right\} = 0$ only if $f_i=0$ for $i \geq 1$. Thus we directly conclude that $P$ is a $\partial_x$-derivative. 
	
	Now suppose that $m \geq 1$, we will show in the next paragraph that $\alpha_{m}=1$. As a consequence the singular differential polynomial
\[
P-\partial_{x}\left(\frac{1}{x^{n}\left(\alpha_{m-1}+1\right)}f\left(u\right)\left(\prod_{k=1}^{m-2}u_{k}^{\alpha_{k}}\right)u_{m-1}^{\alpha_{m-1}+1}\right)\quad{\rm if}\,m\geq2
\]
or
\[
P-\partial_{x}\left(\frac{1}{x^{n}}\sum_{i\geq0}\frac{f_{i}}{i+1}u_{0}^{i+1}\right)\quad{\rm if}\,m=1
\]
has a smaller lexicographic order than $P$. Moreover, this singular differential polynomial satisfies the three hypothesis of our lemma. Therefore, after repeating our steps  a finite number of times, we get that $P$ is a $\partial_{x}$-derivative. 

In order to show $\alpha_{m}=1$, we introduce the following bracket between singular differential polynomials 
\[
\left[f,g\right]=\sum_{s\geq0}\left(\left(\partial_{x}^{s}f\right)\frac{\partial g}{\partial u_{s}}-\left(\partial_{x}^{s}g\right)\frac{\partial f}{\partial u_{s}}\right).
\]
It is not the commutator of the star product, this bracket notation will only be used throughout this proof. It is clearly explained in the proof of \cite[Lemma 2.5]{Buryak_Hodge15} how the condition $\left\{ \int Pdx,\int\frac{u^{3}}{6}dx\right\} =0$ implies 
\[
\int\left[uu_{1},P\right]dx=0, 
\]
and the argument holds here for the exact same reasons. We deduce that  $\left[uu_{1},P\right]$ is a $\partial_{x}$-derivative. Moreover, we have

\begin{equation}\label{eq: uu1}
\begin{split}
\left[uu_{1},P\right] & =\left(\sum_{k=1}^{m}\left(k+1\right)\alpha_{k}-\alpha_{1}-1\right)\frac{f\left(u_{0}\right)u_{1}^{\alpha_{1}}\cdots u_{m}^{\alpha_{m}}}{x^{n}}\times u_{1}\\
 & +{\rm monomials\,of\,lower\,lexicographic\,order.}
\end{split}
\end{equation}

The only exception being whenever the term in parenthesis vanishes, that is when $m=1$ and $\alpha_{m}=1$, but in this case we directly get the desired conclusion $\alpha_m=1$. Now, since $\left[uu_{1},P\right]$ is a $\partial_{x}$-derivative we deduce from Eq. (\ref{eq: uu1}) that $\alpha_{m}=1$ when $m\geq2$. This concludes the proof.
\end{proof}

\bibliographystyle{alpha}
\bibliography{DRk}

\end{document}